\documentclass[reqno,12pt]{amsart}
\usepackage[nokeys]{style}
\usepackage{style}

\begin{document}

\title[Heisenberg products: commutative and non commutative contexts.]
{The Heisenberg product: from Hopf algebras and species to symmetric functions}

\author{Marcelo Aguiar}
\email{maguiar@math.cornell.edu}
\author{Walter Ferrer Santos}
\email{wrferrer@cmat.edu.uy}
\author{Walter Moreira}
\thanks{}
\email{walter@waltermoreira.net}

\begin{abstract} Many related products and coproducts (e.g. Hadamard, Cauchy, Kronecker, induction, internal, external, Solomon, composition, 
Malvenuto--Reutenauer, convolution, etc.) have been defined in the following objects : species, representations of the symmetric groups, 
symmetric functions, endomorphisms of graded connected Hopf algebras, permutations, non--commutative symmetric functions, quasi--symmetric functions, etc. 
With the purpose of simplifying and unifying this diversity we introduce yet, another --non graded-- product the \emph{Heisenberg product}, that for 
the highest and lowest degrees produces the classical external and internal products (and their namesakes in different contexts). In order to define it, we start from the two opposite more general extremes: species in the ``commutative context'', and endomorphisms of Hopf algebras 
in the ``non--commutative'' environment. Both specialize to the space of commutative symmetric functions where 
the definitions coincide. We also deal with the different coproducts that these objects carry --to which we add the 
\emph{Heisenberg coproduct} for quasi--symmetric functions--, 
and study their Hopf algebra compatibility particularly for symmetric and non commutative symmetric functions. 
We obtain combinatorial formulas for the structure constants of the new product that extend, generalize and unify results 
due to Garsia, Remmel, Reutenauer and Solomon. In the space of quasi--symmetric functions, we describe explicitly the new operations in terms of alphabets.
\end{abstract}
       \maketitle
\newpage
       \tableofcontents
            \newcommand\textbox[2]{%
  \parbox{#1}{\tiny\sffamily%
    {\begin{center}#2\end{center}}}}
    
\section{Introduction}

\subsection{General  description of the paper}

The diagram in Figure~\ref{F:spaces} displays the array of spaces we consider in this work and the table in Figure~\ref{F:table-of-names} displays 
the variety of products that had been therein defined and given due consideration by many authors. 

In all these levels we define a new product, that we call the \emph{Heisenberg 
product} that is not graded and that at the smallest and larger 
degrees produce respectively the internal and external products 
(or their namesakes as depicted in Figure~\ref{F:table-of-names}). This is what we call 
the {\emph interpolation property} of the Heisenberg product.

The name we choose for the new product, comes from the one that it usually carries at the level of Hopf algebras.  
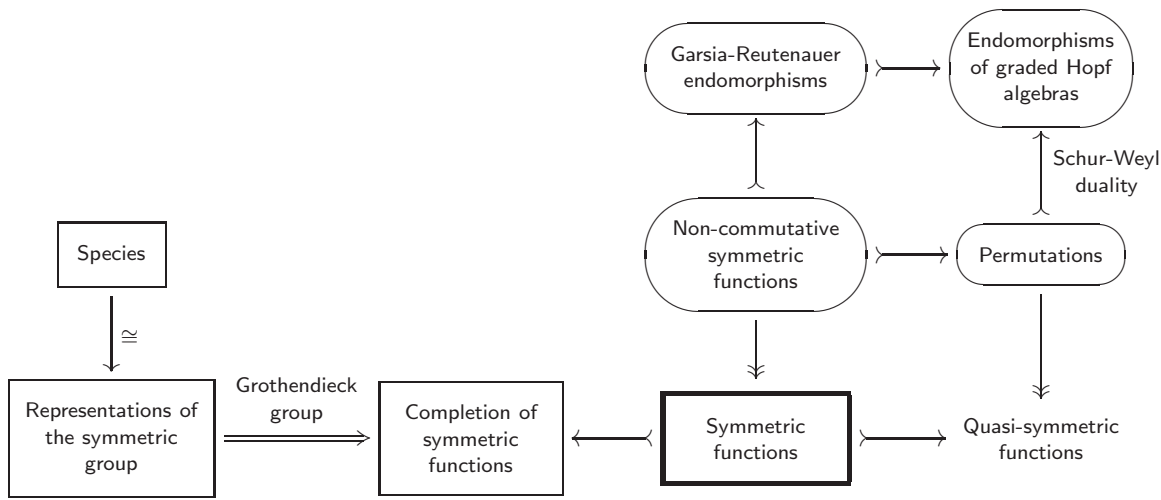
\begin{figure}[!ht]
  \begin{equation*}%
    \xymatrix{%
      &&
      &
      \ovalbox{\textbox{2.5cm}{Garsia-Reutenauer endomorphisms}}
      \ar@{ >->}[r]
      &
      {\cornersize*{40pt}
        \ovalbox{\textbox{2cm}{Endomorphisms\\ of graded Hopf\\ algebras}}}
      \\
      \fbox{\textbox{1cm}{Species}}
      \ar@{-> }[d]^-{\cong}
      &&
      &
      \ovalbox{\textbox{2.5cm}{Non-commutative symmetric \\ functions}}
      \ar@{ >->}[u]
      \ar@{ >->}[r]
      \ar@{->> }[d]
      &
      \ovalbox{\textbox{1.8cm}{Permutations}}
      \ar@{ >->}[u]_-{\textbox{1.5cm}{Schur-Weyl\\ duality}}
      \ar@{->> }[d]
      \\
      \fbox{\textbox{2.3cm}{Representations of the symmetric\\ group}}
      \ar@{=>}[rr]^-{\textbox{2cm}{Grothendieck\\ group}}
      &&
      \fbox{\textbox{2cm}{Completion of\\ symmetric\\ functions}}
      &
      {\setlength\fboxrule{1.5pt}%
        \fbox{\textbox{2cm}{Symmetric functions}}}
      \ar@{ >->}[r]
      \ar@{ >->}[l]
      &
      \textbox{2.3cm}{Quasi-symmetric functions}
    }
  \end{equation*}
  \caption{Relation among the spaces where the Heisenberg structure is
    introduced.}
  \label{F:spaces}
\end{figure}

Observe that we divided the spaces into three groups, marked with different
kinds of boxes in the diagram.  In the first group, marked
with square boxes, we start by constructing the Heisenberg product in the category of
species --or equivalently of representations of the symmetric groups, and we translate it to the space of symmetric functions.  We
use the Grothendieck group functor and the Frobenius characteristic map to move from categories to symmetric functions.  This work is covered in
Sections~\ref{S:species} and~\ref{S:symmetric}. In the space of symmetric functions we give an explicit combinatorial formula for the Heisenberg product in 
the linear basis of the complete homogeneous symmetric functions. Later this formula will be the main tool used to relate the commutative with the non commutative
 environment in Section~\ref{S:nc-to-commutative}. We also find an explicit formula for the structure constants of the Heisenberg product in the basis of 
power sums.  The existence of these simple formul\ae\/\, poses the problem of finding an explicit description for the structure constants on the basis of 
Schur functions. The answer would contain as extreme cases the Littlewood-Richardson rule and (a still unknown) rule for the Kronecker coefficients.

In the second group, marked with oval boxes, we start by considering the Heisenberg
product in the space of endomorphisms of a graded Hopf algebra in 
Section~\ref{S:endomorphisms}.  This product is well known in the
theory of Hopf algebras. In Section~\ref{S:permutations} we restrict it, 
via Schur-Weyl duality, to
the vector space linearly generated by permutations, 
on one side, and on the other side to what we define as the space of
descents (or Garsia--Reutenauer) endomorfisms of a Hopf algebra in Section~\ref{S:descents}.  These two constructions are confluent when we restric 
both to the space of non-commutative symmetric functions in 
Section~\ref{S:ncsymmetric}. Moreover, in this Section 
we present two different proofs of the restriction of the 
Heisenberg product to the space
of descents (non commutative symmetric functions) 
--and suggest two others, relating to work of Brown, Mahajan, Schocker and 
others--. The first proof 
 follows along the line of results of Garsia and Reutenauer characterizing 
the descents via Schur--Weyl duality. In this perspective, the invariance 
of the space of descents 
within the space of permutations with respect to the Heisenberg product, 
becomes a result in Hopf algebra theory. The second proof that has a 
combinatorial content, is based upon an explicit calculation of the 
Heisenberg product of elements of the cannonical basis $X_\alpha$, 
where $\alpha$ is a composition.         
By the interpolation property of the Heisenberg product, 
the combinatorial formula thus obtained, contains as special cases 
rules for the product in 
Solomon's descent algebra and for the external product of two basis 
elements of $\Sigma$. One readily verifies that the former is precisely the
well-known rule of Garsia, Remmel, Reutenauer, and Solomon, and the latter 
is just the usual formula for the external product of the basic elements 
of the space of descents. 
In that manner, the apparently disconnected combinatorial descriptions of 
two different products 
--the Solomon and the external product--, become  connected by the 
introduction of the new Heisenberg product.

The square boxes and the oval boxes have as common ground the bold square 
box that contains the space of symmetric functions.  
We show in Section~\ref{S:nc-to-commutative} that via the usual projection 
of non-commutative
symmetric functions onto symmetric functions the 
two definitions of Heisenberg product coincide in the
latter space .  
This --together with the formal similarity of the constructions-- 
justifies the use of the same name in both contexts.  

Moreover, the spaces of non-commutative and
commutative symmetric functions carry a well-known coalgebra
structure.  We show that this structure is compatible with the Heisenberg
product, producing in both cases related Hopf algebras. This appears in 
Section~\ref{S:hopf} and~\ref{S:isomorphisms}.

Finally, the remaining vector space  appearing in Table \ref{F:spaces}, 
is the space of quasi-symmetric
functions, which is dual to the space of non-commutative symmetric functions.
We construct by means of this duality, the Heisenberg coproduct in
quasi-symmetric functions in Section~\ref{S:quasi-symmetric}. 
In this situation, this coproduct together with the usual product, 
endow the quasi--symmetric functions with a Hopf algebra structure. 

In the Appendix, we present the proofs of three technical lemmas, 
that in order to spare the reader of some nonessential distractions, 
were omitted in the places that the statements  appeared.

\subsection{Terminology and general notations}

In all the spaces we consider there are at least two well-known
products.  Although they are closely related by the inclusions,
projections, and isomorphisms in Figure~\ref{F:spaces}, 
mathematical developments have 
given them non consistent names in many cases.

Figure~\ref{F:table-of-names} shows the more standardt nomenclature.
\renewcommand\arraystretch{1.5}
\begin{figure}[ht]
  \tiny
  \centering
  \begin{tabular}{|D||C|C|C|C|C|C|}%
    \hhline{~|------|}\multicolumn{1}{c|}{}
    & \textsf{Species}
    & \textsf{Representations}
    & \textsf{Symmetric functions}
    & \textsf{Non-commutative symmetric functions}
    & \textsf{Permutations}
    & \textsf{Endomorphisms of graded Hopf algebras}
    \tabularnewline\hhline{-|=|=|=|=|=|=|}
    \textsf{internal product}
    & Hadamard\\ ($\times$)
    & Kronecker\\ ($\internalsym$)
    & internal\\ ($\internalsym$)
    & Solomon\\ ($\internalsym$)
    & composition\\ ($\internalperm$)
    & composition\\ ($\internalend$)
    \tabularnewline\hline
    \textsf{external product}
    & Cauchy\\ ($\tensorprod$)
    & induction\\ ($\externalsym$)
    & external\\ ($\externalsym$)
    & external\\ ($\externalsym$)
    & Malvenuto-Reutenauer\\ ($\externalperm$)
    & convolution\\ ($\externalend$)
    \tabularnewline\hline
  \end{tabular}
  \caption{Standard terminology and symbols for the products}
  \label{F:table-of-names}
\end{figure}

The new product we
introduce in this article, will be called in all the different contexts, the 
{\em Heisenberg product}. 

We work over a field of characteristic zero, that will be denoted as $\field$.

When dealing with results in {\em Hopf algebra theory}, we adopt the usual notations in the area as presented for example in \cite{montgomery93}. 
We adopt Sweedler's convention for the comultiplication and write 
$\Delta(h)=\sum h_1 \otimes h_2$, moreover $S$ in general will be the antipode. An element of the Hopf algebra is said to be primitive if $\Delta(h)=h\otimes 1 + 1 \otimes h$. The space of primitive elements is denoted as $\operatorname{Prim}(H)$.  

The group of the permutations of $n$ elements will be denoted as $S_n$, $[n]:=\{1,\cdots,n\}$ and a permutation  $\sigma:[n] \rightarrow [n]$ 
frequently will be written in word format as: $\sigma(1),\cdots,\sigma(n)$.  

If $(a_1,\cdots,a_r)$ is a composition of $n$ we write $(a_1,\cdots,a_r)\models n$. If we allow some of the elements $a_i$ to be zero, we call it a \emph{weak composition} or a pseudo composition. Two compositions  with the same elements but in a different order are considered as different compositions 
but the same partition. 

For the basic notations and results in the theory of {\em representations of symmetric groups} we refer the reader to Zelevinski's \cite{zelevinski81}.  

      \part{The commutative context}

            \section{The Heisenberg product of species}
\label{S:species}

We introduce the notion of {\em Heisenberg product} of two species,
generalizing the ordinary Cauchy and Hadamard products. We follow the
notation and terminology of~\cite{aguiar04} and~\cite{bergeron98}.

Let $\Setcross$ be the category of finite sets with bijections among
them as morphisms. Let $\Vect$ or $\Vect_\field$ be the category of vector spaces over a fixed field.

\begin{definition}
  The \emph {category of species}, denoted as
  $\Sp$, is the abelian category  of functors from $\Setcross$ into $\Vect$: $\Sp= \Vect^{\Setcross}$ --with abelian structure induced by that of $\Vect$--.  An object  $\p:\Setcross\to\Vect$ of $\Sp$ is called a {\em species}. The elements of $\Sp(\p,\q)$ --i.e. the morphisms of $\Sp$--, are the natural transformations between $\p$ and $\q$. The evaluation of the species $\p$ over a set $I$ is a vector space denoted as $\p[I]$, and for a bijection $f$ the effect of the species $\p$ is denoted as $\p[f]$. The image of the set $[n] = \{1,2,\ldots, n\}$ is written, for simplicity, $\p\bigl[ [n] \bigr] = \p[n]$, and by
  convention $\p[0]=\p[\emptyset]$.
\end{definition}

For future reference we consider the following species: 
\begin{enumerate}
\item The {\em species one} $\I$: $\I[\emptyset]=\field$ and $\I[I]=0$ for $I\neq\emptyset$,
\item The {\em exponential species} $\e$: $\e[I]=\field$ for $I\neq\emptyset$.
\end{enumerate}
\begin{definition}
  The {\em Heisenberg product of species} is the functor $\smashprod:\Sp\times\Sp\to\Sp$
  given~by:
  \begin{equation}
    \label{D:smash-species}
    (\p\smashprod\q) [I]
    =
    \bigoplus_{I=S\cup T} \p[S] \otimes q[T].
  \end{equation}
  Given two finite sets $I$ and $J$, and a bijection $f:I\to J$, we
  obtain a bijection $(S,T) \mapsto \bigl( f(S), f(T) \bigr)$ between
  pairs $(S,T)$ with $I=S\cup T$, and pairs $(S',T')$ with
  $J=S'\cup T'$. The map $(\p\smashprod\q)[f]: (\p\smashprod\q)[I] \rightarrow (\p\smashprod\q)(J)$, induced
  by the maps
  \begin{equation*}
    \p[f_{|_S}]\otimes \q[f_{|_T}]
    :
    \p[S] \otimes \q[T] \to \p[S'] \otimes \q[T']
  \end{equation*}
  defines $\p\smashprod\q$ for arrows. Hence, $\p\smashprod\q \in \Sp$.
\end{definition}

This definition involves in particular two special situations: the case of the Cauchy product and of the Hadamard product of species. 

Indeed, consider the two ``limit'' situations in terms of the size of the intersection $S \cap T$ in the above sum: 
\begin{align*}
  (\p \tensorprod \q) [I] 
  &= 
  \bigoplus_{\substack{%
      I=S\cup T \\
      \emptyset = S\cap T}}\p[S]\otimes \q[T]=\bigoplus_{I=S\sqcup T}\p[S]\otimes \q[T],
  \\
  (\p \hadamardprod \q) [I] 
  &=\bigoplus_{\substack{%
      I=S\cup T \\
      S\cap T=I}}\p[S]\otimes \q[T]=
  \p[I] \otimes \q[I],
\end{align*}
where the symbol $S\sqcup T$ stands for the {\em disjoint union} of $S$ and $T$.

Since we can write:
\begin{equation*}
  (\p \smashprod \q) [I]
  =
  (\p \hadamardprod \q) [I]
  +
  \bigoplus_{\substack{%
      I=S\cup T \\
      \emptyset \neq S\cap T \neq I}}
  \p[S]\otimes \q[T]      
  +
  (\p \tensorprod \q) [I],
\end{equation*}
we say informally that the Heisenberg product ``interpolates'' between the Cauchy  and the Hadamard products --$\p \tensorprod \q$ and $\p \hadamardprod \q$ respectively--.

The triples $(\Sp,\tensorprod, \I)$ and $(\Sp,\hadamardprod, \e)$ endowed with the obvious associative constraint, are additive monoidal categories with $\I$ and $\e$ as unit objects. 

In the next theorem we present a generalization of this fact proving that $(\Sp, \smashprod, \I)$ is also monoidal and additive. 

\begin{theorem}
  The functor $\smashprod:\Sp\times\Sp \to \Sp$ together with the natural associativity constraint and with the unit object $\I$, endows the category~$\Sp$, with an additive monoidal structure.  
\end{theorem}

\begin{proof}
  Let $\p$, $\q$, and $\r$ be three species. We prove that
  $
  \bigl((\p\smashprod\q) \smashprod\r\bigr)[I] =
  \bigl(\p \smashprod (\q \smashprod \r) \bigr) [I]
  $
  for all finite sets $I$.

  We have
  \begin{align}
    \label{E:smash-species-assoc-left}
    \bigl((\p\smashprod\q) \smashprod\r\bigr)[I]
    &=
    \bigoplus_{I=S\cup T} (\p\smashprod\q)[S]\otimes \r[T]
    =
    \bigoplus_{I=S\cup T}\bigoplus_{S=U\cup V}
    \p[U]\otimes \q[V]\otimes\r[T],
    \\
    \label{E:smash-species-assoc-right}
    \bigl(\p \smashprod (\q \smashprod \r) \bigr) [I]
    &=
    \bigoplus_{I=S'\cup T'} \p[S'] \otimes (\q\smashprod\r)[T']
    =
    \bigoplus_{I=S'\cup T'} \bigoplus_{T'=U'\cup V'}
    \p[S'] \otimes \q[U'] \otimes \q[V'],
  \end{align}
  which implies that ~\eqref{E:smash-species-assoc-left}
  and~\eqref{E:smash-species-assoc-right} coincide.  From the
  definition~\eqref{D:smash-species} it is clear that $\p\smashprod\I =
  \I\smashprod\p = \p$.
\end{proof}

In addition to the relation via interpolation of the Cauchy, Hadamard and Heisenberg products, they are also related by a natural isomorphism:
\begin{equation}
  \label{E:isomorphism-smash-composition}
  (\p\cdot\e)\times(\q\cdot\e)
  \cong
  (\p\smashprod\q)\cdot\e,
\end{equation}

The required isomorphism~\eqref{E:isomorphism-smash-composition} appears naturally when we
evaluate each side on a finite set $I$.  We have:
\begin{align*}
  \bigl((\p\cdot\e)\times(\q\cdot\e)\bigr)
  [I]
  &=
  \biggl(
  \bigoplus_{I=S\sqcup T}
  \p[S]\otimes\field
  \biggr)
  \otimes
  \biggl(
  \bigoplus_{I=S'\sqcup T'}
  \q[S']\otimes\field
  \biggr)
  \\
  \bigl((\p\smashprod\q)\cdot\e\bigr)[I]
  &=
  \bigoplus_{I=J\sqcup K}
  (\p\smashprod\q)[J]
  \otimes
  \e[K]
  =
  \bigoplus_{J\subseteq I}
  \biggl(
  \bigoplus_{J=S\cup S'}
  \p[S]\otimes\q[S']
  \biggr)
  \otimes
  \field,
\end{align*}
and clearly both spaces are naturally isomorphic.

\subsection{The Heisenberg product of representations of the symmetric group}

The language of species and the more classical language of
representations of the symmetric group are essentially the same --see for example \cite{bergeron98}--.  Using the equivalence of the categories, we translate the Heisenberg
product to the category of representations.  

Therefore, we obtain an
associative product at the level of the representations of the symmetric groups, which interpolates between the Kronecker and the induction product.

Let $\Rep(\S_n)$ be the category whose objects are finite dimensional
representations of $\S_n$ and whose morphisms are $\S_n$-module
homomorphisms.  We consider the category
\begin{equation}
  \label{E:def-representations-category}
  \representations 
  = 
  \prod_{n\ge 0} \Rep(S_n).
\end{equation}
\begin{theorem} [\cite{bergeron98}]
  \label{T:functor-Sp-Lambda}
  The functor $\functor{F}:\Sp\to\representations$ given by
  \begin{equation*}
    \functor{F}(\p)
    = 
    \bigl( \p[0], \p[1],\ldots, \p[n],\ldots\bigr)
  \end{equation*}
  is an equivalence of abelian categories.  The action of the permutation
  group $S_n$ in the space $\p[n]$ is given by $\sigma\cdot
  x=\p[\sigma](x)$.
\end{theorem}

Using the functor $\functor{F}$ we can express the Heisenberg product in
terms of representations.  The explicit construction uses the
induction and restriction of representations to certain subgroups
defined as follows.

Let $p$, $q$ be non-negative integers. Given permutations $\sigma\in
\Perm_p$ and $\tau\in \Perm_q$, let $\sigma\times\tau\in \Perm_{p+q}$
be the permutation
\begin{equation}
  \label{D:definition-times}
  (\sigma\times\tau)(i)
  =
  \begin{cases}
    \sigma(i) 
    & 
    \text{ if $1\leq i\leq p$,}
    \\
    \tau(i-p)+p
    &
    \text{ if $p+1\leq i\leq p+q$.}
  \end{cases}
\end{equation}
This operation gives an embedding of $\Perm_p\times\Perm_q$ into
$\Perm_{p+q}$ called the \emph{parabolic embedding}.

Let $n$ be an integer satisfying $\max(p,q)\le n\le p+q$. Define the
set $\Perm_p\times_n\Perm_q=\Perm_{n-q} \times \Perm_{p+q-n} \times
\Perm_{n-p}$, and consider the embeddings
\begin{align}
  \label{E:embed-times-sub-n-in-Sn}
  \S_p \times_n \S_q \mono \S_n,
  & \quad 
  (\sigma,\rho,\tau)\mapsto \sigma\times\rho\times\tau,
  \\
  \label{E:embed-times-sub-n-in-times}
  \S_p \times_n \S_q \mono \S_p\times\S_q,
  & \quad
  (\sigma,\rho,\tau)\mapsto (\sigma\times\rho,\rho\times\tau).
\end{align}

\begin{definition}
  The {\em Heisenberg product of representations} is the functor $\smashprod:
  \representations\times \representations \to \representations$ defined for
  $V\in\Rep(\S_p)$ and $W\in\Rep(\S_q)$ as
  \begin{equation}
    \label{E:definition-smash-representations}
    V \smashprod W 
    = 
    \bigoplus_{n=\max(p,q)}^{p+q} 
    \Ind_{\S_p\times_n \S_q}^{\S_n}
    \Res_{\S_p\times_n\S_q}^{\S_p\times\S_q} (V\otimes W).
  \end{equation}
\end{definition}

Let $(V\smashprod W)_n$ denote the component of degree $n$
in~\eqref{E:definition-smash-representations} and the top
component, that is, when $n=p+q$.  In this case, the
embedding~\eqref{E:embed-times-sub-n-in-times} is the identity
and~\eqref{E:embed-times-sub-n-in-Sn} is the standard parabolic
embedding $\S_p \times \S_q \mono \S_{p+q}$. Then, we have that:
\begin{equation*}
  (V\smashprod W)_{p+q}
  =
  \Ind_{\S_p\times \S_q}^{\S_{p+q}}(V\otimes W),
\end{equation*}
which is the usual induction product of
representations~\cite{geissinger77,zelevinski81}.

On the other hand, when $n=p=q$, the 
embedding~\eqref{E:embed-times-sub-n-in-Sn} is the identity
and~\eqref{E:embed-times-sub-n-in-times} is the diagonal embedding
$\S_n \mono \S_n\times\S_n$.  Therefore,
\begin{equation*}
(V\smashprod W)_n = \Res^{\S_n\times \S_n}_{\S_n}(V\otimes W),
\end{equation*}
which is the Kronecker product of
representations~\cite{geissinger77,zelevinski81}.

The Heisenberg product contains terms of intermediate degrees between
$\max(p,q)$ and $p+q$; in this sense it ``interpolates'' between the
Kronecker and induction products.  It is a remarkable fact  that, as the
Kronecker and induction products, the Heisenberg product is associative. 

Moreover, it can be lifted to other settings ( permutations, non-commutative symmetric functions, and dually, quasi-symmetric functions) and for all the 
instances of our construction of the new product, these properties hold.

The next theorem proves that the operation defined
in~\eqref{E:definition-smash-representations} is the translation of
the Heisenberg product in species defined in~\eqref{D:smash-species}.  Note
that although the functor $\functor{F}$ is an equivalence of
categories, the language of species is considerably cleaner than the
language of representations.  The lengthy verifications in
Theorem~\ref{T:equivalence-species-representations} ilustrate this
claim.

\begin{theorem}
  \label{T:equivalence-species-representations}
  The Heisenberg product of representations endows
  $\representations$ with an additive monoidal structure  with unit objet  the representation --denoted as $\I$-- that is $\field$ in degree zero and zero elsewhere. Moreover, the functor
  $\functor{F}$ given in Theorem~\ref{T:functor-Sp-Lambda} is monoidal, i.e. 
\begin{equation}
    \label{E:F-preserves-smash}
    \functor{F}(\p\smashprod\q)
    \cong
    \functor{F}(\p) \smashprod \functor{F}(\p)
    \quad,\quad
    \functor{F}(\I)
    \cong
    \I
  \end{equation}
  for species $\p$ and $\q$.
\end{theorem}

\begin{proof}
It is enough to verify that the Heisenberg product of representations
defined by~\eqref{E:definition-smash-representations}
satisfies~\eqref{E:F-preserves-smash}, since $\functor{F}$ is
already an equivalence of categories and $\Sp$ is a tensor category
with the Heisenberg product and with unit $\I$.

Fix $i$, $j$, and $n$, three non-negative integers such that
$\max(i,j)\le n\le i+j$. We claim that we have an isomorphism in
$\representations$:
\begin{equation}
  \label{E:iso-species-induction-restriction}
  \bigoplus_{
    \substack{
      [n]=S\cup T \\
      \card S = i \\
      \card T = j}}
  \p[S] \otimes \q[T]
  \cong
  \Ind_{\S_i\times_n\S_j}^{\S_n}
  \Res_{\S_i\times_n\S_j}^{\S_i\times\S_j}
  \bigl( \p[i] \otimes \q[j] \bigr).
\end{equation}
Once this isomorphism is established, taking the direct sum over $i$
and $j$, we obtain the $n$-th coordinate of
$\functor{F}(\p\smashprod\q)$ in the left hand side, and the $n$-th
coordinate of the product $\functor{F}(\p)\smashprod \functor{F}(\q)$
on $\representations$ in the right hand side.

The following fact can be proved easily. Let $A$ and $B$ be finite
totally ordered sets. Given decompositions
$A=A_1\sqcup\dots\sqcup A_n$
and
$B=B_1\sqcup\dots\sqcup B_n$,
with $\card A_i=\card B_i$ for $i=1,\dots,n$, there is only one
bijection $f:A\to B$ such that $f(A_i) = B_i$ and $f_i =
f_{|_{A_i}}:A_i\to B_i$ is increasing, for all $i=1,\dots,n$. We call
$f$ the {\em canonical} bijection between $A$ and $B$ induced by the
partitions.

To establish the isomorphism we consider the following definitions.  Given
$S$ and $T$ such that $[n]=S\cup T$, let $S'=S\setminus T$ and
$T'=T\setminus S$. If $\card S=i$ and $\card T=j$, then let
$f_{S,T}:[n]\to [n]$ be the canonical bijection induced by the
following partitions of~$[n]$:
\begin{equation*}
  S'\sqcup (S\cap T) \sqcup T'
  \quad\text{and}\quad
  [n-j]\sqcup [n-j+1,i]\sqcup [i+1,n]
\end{equation*}
and let $f_{S'}$, $f_{S\cap T}$, and $f_{T'}$, be the restriction to
the corresponding subsets.  From the monotonicity conditions for
$f_{S,T}$, we get that $f_{S,T}^{-1}$ belongs to $\S_i\times_n\S_j$.

We consider the standard identification of the induction module
$\Ind_H^G(V)$ with the tensor product $\field G\otimes_{\field H} V$.  Let $u\in\p[S]$
and $v\in\q[T]$, and define the map
\begin{equation}
\begin{aligned}
  \label{E:def-of-psi}
    \p[S] \otimes \q[T]
    &
    \stackrel{\psi}{\longrightarrow}
    \Ind_{\S_i\times_n\S_j}^{\S_n}
    \Res_{\S_i\times_n\S_j}^{\S_i\times\S_j}
    \bigl( \p[i] \otimes \q[j] \bigr)
    \\
    u\otimes v
    &
    \longmapsto
    f_{S,T}^{-1} \otimes 
    \bigl( \p[f_{S'}\sqcup f_{S\cap T}] (u)
           \otimes
           \q[f_{S\cap T}\sqcup f_{T'}] (v)
    \bigr)
\end{aligned}
\end{equation}
and extend it to the direct sum
in~\eqref{E:iso-species-induction-restriction}.

For a permutation $\sigma\in \S_n$, the action of $\sigma$ in
$u\otimes v$ is, according to Theorem~\ref{T:functor-Sp-Lambda},
\begin{equation}
  \label{E:Sn_acts_tensor}
  \sigma\cdot (u\otimes v) 
  =
  \p[\sigma_{|_S}](u)\otimes \q[\sigma_{|_T}](v).
\end{equation}
Observe that $\sigma\cdot(u\times v)\in \p\bigl[ \sigma(S) \bigr]
\otimes \q\bigl[ \sigma(T) \bigr]$. The application of the map $\psi$
yields
\begin{equation}
  \label{E:psi_of_the_action}
  \psi\bigl( \sigma\cdot(u\otimes v) \bigr)
  =
  f_{\sigma(S),\sigma(T)} \otimes
    \bigl( \alpha(u)
           \otimes
           \beta(v)
    \bigr)
\end{equation}
where 
$
\alpha=\p[ f_{\sigma(S')} \sqcup f_{\sigma(S)\cap\sigma(T)} ] 
\p[\sigma_{|_S}]
$
and 
$
\beta=\q[ f_{\sigma(S)\cap
  \sigma(T)} \sqcup f_{\sigma(T')} ] \q[\sigma_{|_T}]
$.
Since we can decompose $\sigma_{|_S}$ into $\sigma_{|_S'} \sqcup
\sigma_{|_{S\cap T}}$, then by the functoriality of $\p$ we get that
\begin{equation*}
  \alpha 
  = 
  \p\bigl[ 
  (f_{\sigma(S')}\sigma_{|_{S'}}) \sqcup 
  (f_{\sigma(S\cap T)}\sigma_{|_{S\cap T}})
  \bigr].
\end{equation*}
Let 
$
\tilde\sigma_{S'}
$
and 
$
\tilde\sigma_{S\cap T}
$
be the only bijections such that the following diagrams commute
\begin{equation}
  \label{E:canonical_bijections_commuting}
  \xymatrix{
    S' \ar[r]^{f_{S'}} \ar[d]_{\sigma_{|_{S'}}}
    & 
    [i] \ar[d]^{\tilde\sigma_{S'}} 
    \\
    \sigma(S') \ar[r]^-{f_{\sigma(S')}} 
    &
    [i]}
  \qquad
  \xymatrix@C=35pt{
    S\cap T \ar[r]^-{f_{S\cap T}} \ar[d]_{\sigma_{|_{S\cap T}}}
    & 
    [n-j+1,i]] \ar[d]^{\tilde\sigma_{S\cap T}} 
    \\
    \sigma(S\cap T) \ar[r]^-{f_{\sigma(S\cap T)}} 
    &
    [n-j+1,i]}
\end{equation}
We conclude that $\alpha$ can be rewritten as 
\begin{equation*}
\alpha 
= 
\p\bigr[ (\tilde\sigma_{S'} \sqcup \tilde\sigma_{S\cap T})
( f_{S'} \sqcup f_{S\cap T} ) \bigr]
=
\p[ \tilde\sigma_{S'} \sqcup \tilde\sigma_{S\cap T} ]
\,
\p[ f_{S'} \sqcup f_{S\cap T} ) ],
\end{equation*}
and we proceed similarly with $\beta$.  

\noindent
In accordance with~\eqref{E:Sn_acts_tensor} we deduce that:
\begin{equation*}
\alpha(u)\otimes \beta(v)
=
(
\tilde\sigma_{S'} \sqcup 
\tilde\sigma_{S\cap T} \sqcup
\tilde\sigma_{T'}
)
\cdot
    \bigl( \p[f_{S'}\sqcup f_{S\cap T}] (u)
           \otimes
           \q[f_{S\cap T}\sqcup f_{T'}] (v)
    \bigr).
\end{equation*}
Note that the permutation
$
\tilde\sigma_{S'} \sqcup 
\tilde\sigma_{S\cap T} \sqcup
\tilde\sigma_{T'}
$
clearly belongs to $\S_i\times_n\S_j$. In
equation~\eqref{E:psi_of_the_action}, since the tensor product of
$f^{-1}_{\sigma(S),\sigma(T)}$ with $\alpha(u)\otimes\beta(v)$ is performed with respect to this subgroup, we can move the permutation to the left
factor where we get
$
 f_{\sigma{S},\sigma{T}}^{-1}
(
\tilde\sigma_{S'} \sqcup 
\tilde\sigma_{S\cap T} \sqcup
\tilde\sigma_{T'})
 = 
\sigma f_{S,T}^{-1}
$.
This equality results again from the
diagrams~\eqref{E:canonical_bijections_commuting}. This is precisely
the definition of the action of $\sigma$ on the image of the map $\psi$.

The map $\psi$ is invertible, since for any element $\sigma\otimes
(x\otimes y)$, we decompose
$\sigma=\xi(\alpha\times\beta\times\gamma)$, where
$\alpha\times\beta\times\gamma\in\S_{n-j}\times\S_{i+j-n}\times\S_{n-i}=\S_i\times_n
\S_j$ and $\xi$ is increasing in the intervals $[n-j]$, $[n-j+1,i]$,
and $[i+1,n]$. Define the disjoint sets 
\begin{equation*}
  A 
  = 
  \xi\bigl( [n-j] \bigr),
  \quad
  B
  =
  \xi\bigl( [n-j+1,i] \bigr),
  \quad
  C
  = 
  \xi\bigl( [i+1,n] \bigr).
\end{equation*}
Then, let $S=A\sqcup B$ and $T=B\sqcup C$. It is straightforward to
find $u\otimes v$ in $\p[S]\otimes \q[T]$ such that $\psi(u\otimes v)
= \sigma\otimes (x\otimes y)$. Similarly, this process applied to the
image of $\psi$ in~\eqref{E:def-of-psi} yields back $u\otimes v$.
\end{proof}

\subsection{The generating function of the Heisenberg product of two species}

The {\em generating function} associated to a species $\p$ is the
formal series
\begin{equation*}
  F_\p(x)
  =
  \sum_{n\ge 0} \dim_\field\p[n]\, \frac{x^n}{n!}.
\end{equation*}
The generating series associated to the Cauchy product
$\p\tensorprod\q$ of two species is the usual (Cauchy) product of the
power series $F_\p$ and $F_\q$.  Similarly, the generating series of
$\p\hadamardprod\q$ is the Hadamard product of the generating series
of $\p$ and $\q$.  Explicitly, if $F_\p(x)=\sum_{n\ge 0}a_n
x^n/n!$ and $F_\q(x)=\sum_{n\ge 0}b_n x^n/n!$, then
\begin{equation*}
  F_{\p\tensorprod\q}(x)
  =
  \sum_{n\ge 0}
  \biggl(
  \sum_{i+j=n} \binom{n}{i}a_ib_j
  \biggr)
  \frac{x^n}{n!}
  \quad\text{and}\quad
  F_{\p\hadamardprod\q}(x)
  =
  \sum_{n\ge0}
  a_nb_n\,
  \frac{x^n}{n!}.
\end{equation*}
The classical names for these operations among formal series justify
the names for the Cauchy and Hadamard products of species.

\begin{theorem}
  The generating function of the Heisenberg product of two species~$\p$
  and $\q$~is
  \begin{equation*}
    F_{\p\smashprod\q}(x)
    =
    \sum_{n\ge0}
    \biggl(
    \sum_{\substack{i,j\le n\\ n\le i+j}}
    \binom{n}{n-i,n-j,i+j-n} a_ib_j
    \biggr)
    \frac{x^n}{n!},
  \end{equation*}
where \[\binom{n}{n-i,n-j,i+j-n}=\dfrac{n!}{(n-i)!(n-j)!(i+j-n)!}.\]
\end{theorem}

\begin{proof}
  In the definition of the Heisenberg product of
  two species --Definition~\eqref{D:smash-species}--, the pairs of sets $(S,T)$ such that
  $[n]=S\cup T$, $\card S=i$, and $\card T=j$, are in bijection with
  the decompositions $[n]=U\sqcup W\sqcup V$ with $\card U + \card W =
  i$ and $\card W+\card V=j$.  Indeed, just take $U=S\setminus T$,
  $V=T\setminus S$, and $W=S\cap T$.
  Clearly, we obtain $\card U=n-i$, $\card V=n-j$, and $\card
  W=i+j-n$.  The multinomial coefficient in the formula of
  $F_{\p \smashprod\q}$ stands precisely for the number of possible
  ways to choose the decomposition $U,V,W$.
\end{proof}



            \section{The Heisenberg product of symmetric functions}
\label{S:symmetric}
\subsection{Species, representations of symmetric groups and symmetric functions}\label{SS:generalities}
In this subsection we recall some basic facts about the relations between species --viewed as such or as representations of the family of all symmetric groups-- and the space of symmetric functions.  

Let $\K(\S_n)$ be the Grothendieck group or representation group of the category of finite dimensional
$\S_n$-modules, and call $\groth$ and $\grothcomplete$ the groups:
\begin{equation*}
\groth
  =
  \bigoplus_{n\ge 0} \K(\S_n) 
  \quad  
  \subseteq
  \quad 
  \grothcomplete
  =
  \prod_{n\ge 0} \K(\S_n).
\end{equation*}

Consider the ring of polynomials $\field[x_1,\cdots,x_n]$ in $n$ variables in which the symmetric group $\S_n$ acts by permuting the variables. 
Call $\Lambda^k_n$ the subring consisting of the homogeneous polynomials of degree $k$ which are invariant under the action of $\S_n$. When $m \geq n$, $\Lambda_m^k$ proyects naturally onto  $\Lambda_n^k$ via the homomorphism $\rho^k_{m,n}: \Lambda_m^k \rightarrow \Lambda_n^k$ which maps the first $n$ variables to themselves, and the other variables to 0. The space $\Lambda^k$ is defined as the inverse limit of the system considered above. 

For $\sym^k$ the \emph{space of symmetric functions} of degree $k$, define
\begin{equation*}
  \sym
  =
  \bigoplus_{k\ge 0}\sym^k
  \quad
\subseteq
\quad
  \symcomplete
  =
  \prod_{k\ge 0}\sym^k,
\end{equation*}
called the \emph{space of symmetric functions} and its \emph{completion},
respectively (see~\cite{macdonald95}).

Observe that $\sym$ and $\symcomplete$ are subspaces of $\field[x_1,x_2,\cdots]$ and $\field \llbracket x_1,x_2,\cdots\rrbracket$ respectively.

We recall the following definition of special elements in $\Lambda$. 
\begin{enumerate}
\item \emph{The elementary symmetric functions} are defined by the generating series:\[\sum_{r\geq 0}e_r(x_1,x_2,\cdots)t^r=\prod_{i \geq 1}(1+x_it),\]
\item \emph{The complete homogeneous symmetric functions} are defined by:
\[\sum_{r\geq 0}h_r(x_1,x_2,\cdots)t^r=\prod_{i \geq 1}(1-x_it)^{-1},\]
\item \emph{The power sums} are defined by:
\[\sum_{r\geq 0}p_r(x_1,x_2,\cdots)t^r=\sum_{i \geq 1}x_i(1-x_it)^{-1}.\]
\end{enumerate}

The above defined functions are elements of $\field \llbracket x_1,x_2,\cdots\rrbracket$, if we want to consider the corresponding elements in 
$\field [x_1,x_2,\cdots x_n]$ we simply set $0= x_{n+1}=x_{n+2}=\cdots$. 

These, can in turn be defined in terms of the \emph{monoidal symmetric functions}.

A \emph{partition} --finite or almost finite-- $\alpha=(a_1,a_2,\cdots)$ with $a_1 \geq a_2 \geq \cdots \geq 0$ determines a monomial $x^\alpha=x_1^{a_1} x_2^{a_2} \cdots$.

The monomial symmetric function associated to $\alpha$ and denoted as $m_\alpha$ is \[m_\alpha=\sum_{\{\widehat{\alpha}\,:\, \widehat{\alpha} \leadsto \alpha\}}x^\alpha,\] where $\widehat{\alpha}$ stands for a composition and the symbol $\widehat{\alpha} \leadsto \alpha$ means that the mentioned composition produces the given partition $\alpha$ by permutation of the entries. For example $m_{(21)} = \sum_{i \neq j} x_i^2x_j$. 

We have the following equalites:

\noindent
\emph{Elementary symmetric functions}:  $e_r=m_{(1^r)}$ where $(1^r)$ is the partition of $r$ formed only by $1$'s.
\\
\noindent
\emph{Complete homogenous symmetric functions}: $h_r=\sum_{\{\alpha: |\alpha|=r\}} m_\alpha$ where the sum is taken over all the partitions of $r$.
\\ 
\noindent
\emph{Power sums}: $p_r=m_{(r)}$, where $(r)$ is the partition $(r,0,0\cdots)$.

For an arbitrary partition $\alpha=(a_1,a_2,\cdots)$ we define:
$e_\alpha=e_{a_1}e_{a_2}\cdots$, $h_\alpha=h_{a_1}h_{a_2}\cdots$, $p_\alpha=p_{a_1}p_{a_2}\cdots$.

When $\alpha$ runs over all partitions, the set of all functions $m_\alpha$ form a $\Z$--basis of $\Lambda$, and the same happens with the sets of the $e_\alpha$'s or the set of the $h_{\alpha}$'s. The set of the $p_\alpha$'s form a $\Q$--basis of $\Lambda_\Q$. 

The Frobenius characteristic map is
the linear isomorphism
\begin{equation*}
  \frobenius:\grothcomplete\otimes_\Z\field\to\symcomplete,
  \quad
  \frobenius(V)
  =
  \frac{1}{n!}
  \sum_{\sigma\in\S_n} \character{V}(\sigma) p_{\cycle(\sigma)},
\end{equation*}
where $V$ is a representation of $\S_n$, $\character{V}$ its character and $p_{\cycle(\sigma)}$ is
the power sum associated the partition of $n$ defined by to the cycle-type of $\sigma$.  The map
$\frobenius$ restricts to an isomorphisms of $\groth$ and $\sym$. See \cite[Proposition I.7.3]{macdonald95} for proofs of the isomorphisms.

The above result, yields another perspective regarding the complete homogeneous symmetric functions. 

Indeed, if $\alpha=(a_1,\ldots,a_r)$ is a composition of $n$ and
\begin{equation*}
  \S_\alpha
  =
  \S_{a_1}\times\dots\times\S_{a_r},
\end{equation*}
it can be viewed as a subgroup of $\S_n$ by
iterating~\eqref{D:definition-times}.  These are the so called \emph{parabolic subgroups} of $\S_n$.  Let $\mathfrak h_\alpha$ denote the permutation representation of $\S_n$ corresponding to the action by multiplication
on the quotient $\S_n/\S_\alpha$.  The isomorphism class of $\mathfrak h_\alpha$ does not depend on the order of the parts of $\alpha$, hence we will consider the representations $\mathfrak h_\alpha$ for $\alpha$ running over the
partitions of $n$.

If we denote the trivial $\S_\alpha$-module by $\trivialrep$ (we omit
the dependence on $\alpha$ for clarity), then the representation
$\mathfrak h_\alpha$ can also be expressed as
\begin{equation}
  \label{E:def-of-h-alpha-as-induction}
  \mathfrak h_\alpha = \Ind_{\S_\alpha}^{\S_n}(\trivialrep).
\end{equation}

The following equality holds $\frobenius(\mathfrak h_\alpha)=h_\alpha$ (see \cite[Proposition I.7.3]{macdonald95}).  

\subsection{The Heisenberg product of complete homogeneous symmetric functions}
\label{SS:complete-homogeneous}

In Section \ref{S:species} and in Subsection \ref{SS:generalities} we have established a path between the objects described in the diagram below:
\begin{equation*}\xymatrix{\Sp \ar[r]^{\mathcal F}& \representations
 \ar@{=>}[r]{}& {\grothcomplete \supset \groth}\ar[r]^-{\otimes_\Z\field}&\groth_\field \ar[r]^{\frobenius}& \Lambda,}
\end{equation*}
where the double arrow means the application of the Grothendieck functor. 


Using the universal property of the Grothendick group functor, it is clear that in order to translate the Heisenberg product from $\representations$ to $\grothcomplete$ it is enough to verify that it is compatible with direct sums. 

It is easy to make this verification in the category of species, where
the colimits are defined pointwisely as the colimits of vector
spaces.  Then, the distributive property of the tensor product with
respect to direct sums shows that $(\p\oplus\q)\smashprod \r =
(\p\smashprod\r)\oplus (\q\smashprod r)$ for $\p\,,\q\,,\r \in \representations$.

Hence the operation $\smashprod$ can be defined in $\grothcomplete$, and the subgroup $\groth$ is clearly  $\smashprod$-- closed since the definition of the Heisenberg product involves only a finite number of summands.



Now, by composition with the Frobenius characteristic isomorphism $\frobenius: \groth_\field \rightarrow \Lambda$ we
obtain an associative product on symmetric functions, which we call
{\em Heisenberg product of symmetric functions}.

The next theorem gives an explicit formula for the Heisenberg product in
the linear basis of $\sym$ formed by the complete homogenous symmetric
functions.  This theorem, besides providing a combinatorial rule 
useful for computations, will later be used to make the connection with the Heisenberg product of non-commutative symmetric functions in
Section~\ref{S:nc-to-commutative}.

In order to express the coefficients of the Heisenberg product of two
complete homogenous symmetric functions, we need to define a
particular set of plane partitions as follows.  Let
$\alpha=(a_1,\ldots,a_r)\composition p$ and
$\beta=(b_1,\ldots,b_s)\composition q$ be two compositions and $n$ an
integer with $\max(p,q)\leq n\leq p+q$. Let $a_0=n-p$, $b_0=n-q$, and
let $\M_{\alpha,\beta}^n$ be the set of all
$(s+1)\times(r+1)$-matrices
\begin{equation*}
  M=(m_{ij})_{
    \substack{0\leq i\leq s\\
      \,0\leq j\leq r}}
\end{equation*}
with non-negative integer entries and such that 
\begin{itemize}
\item the sequence of column sums  is $(a_0,a_1,\ldots,a_r)$,
\item the sequence of row sums  is $(b_0,b_1,\ldots,b_s)$,
\item the first entry is $m_{00}=0$.
\end{itemize}
We illustrate these conditions as follows:
\begin{equation*}
    \begin{smasharray}{cccc!{\putbar}c}
      0 & m_{01} & \cdots & m_{0r} & n-q \\ 
      m_{10} & m_{11} & \cdots & m_{1r} & b_1 \\
      \vdots & \vdots & \ddots & \vdots & \vdots \\
      m_{s0} & m_{s1} & \cdots & m_{sr} & b_s \\ 
      \cline{1-4} 
      n-p & a_1  & \cdots &  a_r \omitbar
    \end{smasharray}
\end{equation*}
Let $p(M)$ be the partition of $n$ whose parts are the non-zero
$m_{ij}$.


\begin{theorem}
  \label{T:lambda-smash-combinatorial}
  There is an associative product $\smashprod$ in $\Lambda$,
  interpolating between the internal and external products, which can
  be expressed in the basis $(h_\alpha)$ of complete homogeneous
  functions as
  \begin{equation}
    \label{E:smash-complete-homogeneous}
    h_\alpha \smashprod h_\beta = \sum_{n=\max(p,q)}^{p+q}
    \sum_{M \in \M_{\alpha,\beta}^n} h_{p(M)}.
  \end{equation}
\end{theorem}

For example, using such theorem we get
\begin{equation*}
  h_{(2,1)}\smashprod h_3 
  = 
  h_{(2,1)}+ h_{(1,1,1,1)}+ h_{(2,1,1)} +
  h_{(2,2,1)}+h_{(2,1,1,1)} + h_{(3,2,1)},
\end{equation*}
where the external product is recognized in the last term and the
internal product in the first one, together with additional terms of
degrees four and five.



The existence of this operation poses the problem of finding an
explicit description for its structure constants on the basis of Schur
functions.  The answer would contain as extreme cases the
Littlewood-Richardson rule and (a still unknown) rule for the
Kronecker coefficients.

\begin{proof}[Proof of Theorem~\ref{T:lambda-smash-combinatorial}]
  We prove that the following formula holds in the category
  $\representations$:
  \begin{equation*}
    \mathfrak h_\alpha \smashprod \mathfrak h_\beta = \bigoplus_{n=\max(p,q)}^{p+q}
    \bigoplus_{M \in \M_{\alpha,\beta}^n} \mathfrak h_{p(M)},
  \end{equation*}
  where the representations $\mathfrak h_\alpha$ are the induced representations
  defined in~\eqref{E:def-of-h-alpha-as-induction}.  The application
  of the Grothendieck group functor and the Frobenius characteristic
  immediately yields~\eqref{E:smash-complete-homogeneous}.

    We fix $n$ in the range $\max(p,q)\le n\le p+q$.  The $n$-summand of
  $\mathfrak h_\alpha\smashprod \mathfrak h_\beta$ is, according
  to~\eqref{E:definition-smash-representations},
  \begin{equation}
    \label{E:nth-term-of-ha-smash-hb}
    (\mathfrak h_\alpha\smashprod \mathfrak h_\beta)_n
    =
    \Ind_{\S_p\times_n\S_q}^{\S_n}
    \Res_{\S_p\times_n\S_q}^{\S_p\times\S_q}
    ( \mathfrak h_\alpha \otimes \mathfrak h_\beta )
    =
    \Ind_{\S_p\times_n\S_q}^{\S_n}
    \Res_{\S_p\times_n\S_q}^{\S_p\times\S_q}
    \Ind_{\S_\alpha\times\S_\beta}^{\S_p\times\S_q}
    (\trivialrep).
  \end{equation}

  Consider the composition of the first two functors
  $
  \Res_{\S_p\times_n\S_q}^{\S_p\times\S_q}
  \Ind_{\S_\alpha\times\S_\beta}^{\S_p\times\S_q}
  $
  in the right hand side
  of~\eqref{E:definition-smash-representations}.  We use Mackey's
  formula to interchange them (see~\cite{weintraub03}), as follows.

  Let $\Upsilon\subset \S_p\times\S_q$ be a complete set of representatives of
  the family of double cosets
  $
  (\S_p\times_n\S_q) \mathbin{\backslash} 
  (\S_p\times\S_q) \mathbin{/} 
  (\S_\alpha\times\S_\beta)
  $.
  \newline
\noindent
For each $v\in\Upsilon$, define
  \begin{equation}
    \label{E:def-v-parabolic}
    {}^\upsilon(\S_\alpha\times\S_\beta) 
    = 
    \upsilon^{-1}(\S_\alpha\times\S_\beta)\upsilon
    \quad\text{and}\quad
    \S_\alpha\times_n^\upsilon\S_\beta
    =
    (\S_p\times_n\S_q) \cap {}^\upsilon(\S_\alpha\times\S_\beta).
  \end{equation}
  The following diagram ilustrates the relative position of these
  groups and subgroups
  
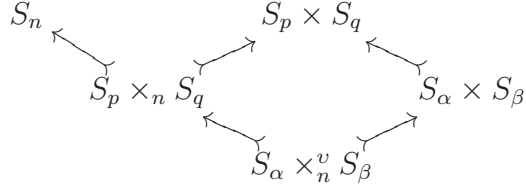
\begin{equation*}
    \xymatrix@C=10pt@R=10pt{ 
      \S_n 
      & & 
      \S_p\times\S_q 
      \\
      & 
      \S_p \times_n \S_q 
      \ar@{ >->}[ul] \ar@{ >->}[ur] 
      & & 
      \S_\alpha \times \S_\beta  
      \ar@{ >->}[ul] 
      \\
      & & 
      \S_\alpha \times_n^\upsilon \S_\beta 
      \ar@{ >->}[ul] \ar@{ >->}[ur]
      & 
    }
  \end{equation*}
  In this situation Mackey's formula reads as the equality
  \begin{equation*}
    \Res_{\S_p\times_n\S_q}^{\S_p\times\S_q}
    \Ind_{\S_\alpha\times\S_\beta}^{\S_p\times\S_q}
    (\trivialrep)
    =
    \bigoplus_{\upsilon\in\Upsilon}
    \Ind_{\S_\alpha\times_n^\upsilon\S_\beta}^{\S_p\times_n\S_q}
    \Res_{\S_\alpha\times_n^\upsilon\S_\beta}^{\S_\alpha\times\S_\beta}
    (\trivialrep).
  \end{equation*}
  Using the transitivity of the induction functor and the property
  that it commutes with coproducts we deduce
  that~\eqref{E:nth-term-of-ha-smash-hb} can be written as
  \begin{equation}
    \label{E:sum-over-double-cosets}
    (\mathfrak h_\alpha\smashprod \mathfrak h_\beta)_n
    =
    \Ind_{\S_p\times_n\S_q}^{\S_n}
    \Res_{\S_p\times_n\S_q}^{\S_p\times\S_q}
    ( \mathfrak h_\alpha \otimes \mathfrak h_\beta )
    =
    \bigoplus_{\upsilon\in\Upsilon}
    \Ind_{\S_\alpha\times_n^\upsilon\S_\beta}^{\S_n}
    (\trivialrep).
  \end{equation}
  In Lemma~\ref{L:bijection-for-ha-smash-hb} we construct a
  bijection $\upsilon\mapsto M_\upsilon$ between $\Upsilon$ and
  $\M_{\alpha,\beta}^n$ with the property that
  $\S_{p(M_\upsilon)}=\S_\alpha\times_n^\upsilon \S_\beta$.
  Then~\eqref{E:sum-over-double-cosets} becomes
  \begin{equation*}
    (\mathfrak h_\alpha\smashprod \mathfrak h_\beta)_n
    =
    \bigoplus_{\upsilon\in\Upsilon}
    \Ind_{\S_\alpha\times_n^\upsilon\S_\beta}^{\S_n}
    (\trivialrep)
    =
    \bigoplus_{\upsilon\in\Upsilon}
    \Ind_{\S_{p(M_\upsilon)}}^{\S_n}(\trivialrep)
    =
    \bigoplus_{M\in\M_{\alpha,\beta}^n}
    \mathfrak h_{p(M)},
  \end{equation*}
  proving the theorem.
\end{proof}

\begin{lemma}
  \label{L:bijection-for-ha-smash-hb}
  In the notations of Theorem~\ref{T:lambda-smash-combinatorial},
  there is a bijection $\Upsilon\cong\M_{\alpha,\beta}^n$ given by
  $\upsilon\mapsto M_\upsilon$, such that
  $
  \S_{p(M_\upsilon)} = \S_\alpha\times_n^\upsilon \S_\beta
  $.
\end{lemma}

\begin{proof}
  The proof of this rather technical lemma is postponed until the appendix~\ref{P:bijection-for-ha-smash-hb}.
\end{proof}

With the help of the coproduct, one can produce a simple relation between the Heisenberg product and the external and internal products in symmetric functions. The coproduct in this situation is dual to the
external product and can be defined on the generators of the basis of complete homogeneous symmetric functions as:
\begin{equation}\label{E:coproductoncompletehom}
  \Delta(h_a) = \sum_{i+j=a} h_i\otimes h_j.
\end{equation}  
The identity that follows was suggested to the authors by A. Zelevinski and does not hold for the space of non-commutative symmetric funcions (see comment after Theorem~\eqref{T:smash-X}).
\begin{lemma}[A. Zelevinski]
 \label{L:zelevinski-identity}
Assume that $f,g \in \Lambda$, then:
\begin{equation}
  \label{E:zelevinski-identity}
  f\smashprod g
  =
  \sum f_1\externalsym (f_2\internalsym g_1)\externalsym g_2,
\end{equation}
where $\Delta(f)=\sum f_1\otimes f_2$ and $\Delta(g)=\sum g_1\otimes
g_2$.
\end{lemma} 

\begin{proof}
The identity~\eqref{E:zelevinski-identity} follows from
formula~\eqref{E:smash-complete-homogeneous}, by collecting the first
row and first column of the matrix $M$ as $(h_\alpha)_1$ and
$(h_\beta)_2$, respectively, and the remaining submatrix of $M$ is
precisely the internal product of the second tensorand $(h_\alpha)_2$ of the coproduct of
$h_\alpha$ with the first tensorand $(h_\beta)_1$ of the coproduct of $h_\beta$.
\end{proof}

\subsection{The Heisenberg product  of power sums}
\label{SS:powersums}
The power sums $(p_\lambda)_{\lambda \vdash n, n\ge 0}$ form a linear
basis of $\Lambda$ over $\Q$.  In this subsection we give an explicit
formula for the Heisenberg product in this basis.

Given two partitions $\lambda$ and $\mu$, denote by $\lambda\mu$ the
concatenation and reordering of $\lambda$ and $\mu$.  For example,
if $\lambda=(3,2,1,1)$ and $\mu=(2,2,1)$, then
$\lambda\mu=(3,2,2,2,1,1,1)$.




\begin{theorem}
  The Heisenberg product in the basis of power sums can be expressed as
  \begin{equation}
    \label{E:smash-power-sums}
    p_\lambda \smashprod p_\mu
    =
    \sum_{\substack{\alpha\gamma=\lambda\\\gamma \beta=\mu}}
    z(\gamma) \; p_{\alpha\gamma\beta},
  \end{equation}
  where $z(\gamma)$ is the order of the stabilizer of the conjugacy
  class of a permutation of cycle-type $\gamma$:
  \begin{equation}
    \label{E:number-z}
    z(\gamma)=\prod_r r^{m_r} m_r!,
  \end{equation}
  being $m_r$ the number of times $r$ occurs in $\gamma$.
\end{theorem}

\begin{proof}
  The coproduct in the basis of power sums is determined by
  requiring the functions $p_n$, with $n$ a non-negative integer, to
  be primitive elements: $\Delta(p_n)=1\otimes p_n + p_n\otimes 1$.
  More explicity,
  \begin{equation*}
    \Delta(p_\lambda) = \sum_{\alpha\beta=\lambda} p_\alpha\otimes p_\beta.
  \end{equation*}
  Then, formula~\eqref{E:zelevinski-identity} reads
  \begin{equation*}
    p_\lambda\smashprod p_\mu
    =
    \sum_{\substack{\alpha_1\alpha_2=\lambda\\\beta_1\beta_2=\mu}}
    p_{\alpha_1}\externalsym (p_{\alpha_2}\internalsym p_{\beta_1})\externalsym p_{\beta_2}.
  \end{equation*}
  But $p_{\alpha_2}\internalsym
  p_{\beta_1}=z(\alpha_2)\delta_{\alpha_2,\beta_1}p_{\alpha_2}$ --see \cite[Chapter I\,(7.12)]{macdonald95}--.  Since the
  external product of power sums is done by concatenating the
  partitions, we obtain the result of the theorem.
\end{proof}

As a particular case, assume that $\lambda$ and $\mu$ are partitions
of $n$.  Note that there is a term in degree $n$ only when
$\lambda=\mu$, otherwise $\gamma$ would never be the empty partition
and the degree of $p_{\alpha\gamma\beta}$ would be strictly greater
than $n$.  Therefore, the only term in degree $n$ is
\begin{equation*}
  \begin{cases}
    z(\lambda)\; p_\lambda,
    &
    \text{if $\lambda = \mu$;}
    \\
    0,
    &
    \text{otherwise;}
  \end{cases}
\end{equation*}
which is the expression of the internal product in the basis of power
sums.

On the other hand for any partitions $\lambda$ and $\mu$, when
$\gamma$ is the empty partition, we obtain the term of largest degree,
namely $p_{\alpha\beta}$, since $z(\gamma)=1$ in this case.  This
gives the external product $p_\lambda\externalsym p_\mu =
p_{\lambda\mu}$.

Note that the coefficients of Formula~\eqref{E:smash-power-sums} in
the basis of power sums are not necessarily the numbers $z(\gamma)$.
Indeed, the partition $\lambda$ may be decomposed, in general, in more
than one way as $\lambda=\alpha\gamma$, since the operation of
concatenation of partitions involves a reordering of the final
result.  For example, let $(1^n)$ be the partitions with $n$ parts
equal to 1.  Then,
\begin{equation}
  \label{E:smash-p_1-to-the-u}
  p_{(1^u)}\smashprod p_{(1^v)}
  =
  \sum_{n=\max(u,v)}^{u+v} 
  \binom{u}{n-v}\binom{v}{n-u}(u+v-n)!\,
  p_{(1^n)}.
\end{equation}
In this case, the partitions of Formula~\eqref{E:smash-power-sums} are
$\alpha=(1^{n-v})$, $\beta=(1^{n-u})$, and $\gamma=(1^{u+v-n})$.  The
number of possible decompositions of $(1^u)$ into two partitions
of length $n-v$ and $u+v-n$ is $\binom{u}{n-v}$, and the same argument
for $(1^v)$ yields the second binomial coefficient.  The remaining
factor of the coefficient is $z(\gamma)=z\bigl((1^{u+v-n})\bigr) =
(u+v-n)!$, according to Formula~\eqref{E:number-z}.

From the explicit expression $h_{(n)}(x)=\sum x_{i_1}^{a_1}x_{i_2}^{a_2}\cdots x_{i_k}^{a_k}$ --where $(a_1,\cdots,a_k)$ ranges over all possible permutations of the parts of $\alpha=(\ell_1,\cdots,\ell_k)$ for all partitions $\alpha$ of $n$--, it is clear that
$h_{(1^u)}=p_{(1^u)}$.  Hence, Formula~\eqref{E:smash-p_1-to-the-u}
can also be deduced from Theorem~\eqref{T:lambda-smash-combinatorial}.
We use this method in~\eqref{E:smash-X_1-to-the-p} for
non-commutative symmetric functions.

      \part{The non-commutative context}

            \section{The Heisenberg product of endomorphisms}
\label{S:endomorphisms}

Let $(H,m,\Delta,\iota,\varepsilon,S)$ be an arbitrary Hopf algebra,
where $m:H\otimes H\to H$ is the product, $\Delta:H\to H\otimes H$ is
the coproduct, $\iota:\field\to H$ is the unit,
$\varepsilon:H\to\field$ is the counit, and $S:H\to H$ is the
antipode.  The space $\End(H)$ of linear endomorphisms of $H$ carries
several associative products.  Let $f,g\in\End(H)$.  Composition and
convolution are respectively defined by the diagrams:
\begin{equation}
  \label{D:composition-convolution-diagram}
  \begin{gathered}
    \xymatrix@R=20pt@C=10pt{
      &  
      {H}
      \ar[rd]^{g} 
      &
      \\
      {H}
      \ar[ru]^{f} \ar[rr]_{g\internalend f} 
      & & 
      {H}
    } 
    \qquad
    \xymatrix@R=20pt@C=25pt{
      {H\otimes H}
      \ar[r]^{f\otimes g} 
      & 
      {H\otimes H}
      \ar[d]^{m}
      \\
      {H}
      \ar[u]^{\Delta}\ar[r]_{f\externalend g} 
      & 
      {H}
    }
  \end{gathered}
\end{equation}

\begin{definition}
  The {\em Heisenberg product of endomorphisms} --denoted as $f \smashprod g$, for $f,g \in \End(H)$-- is defined by the diagram:
  \begin{equation}
    \label{D:smash-diagram}
    \begin{gathered}
      \xymatrix@R=20pt@C=10pt{
        &  
        H^{\otimes 3}
        \ar[rr]^{\mathrm{cyclic}} 
        &&
        H^{\otimes 3}
        \ar[rd]^{1\otimes m}
        \\
        H^{\otimes 2}
        \ar[ru]^{\Delta\otimes 1} 
        & && & 
        H^{\otimes 2}
        \ar[d]^{1\otimes g}
        \\
        H^{\otimes 2}
        \ar[u]^{f\otimes 1}
        & && & 
        H^{\otimes 2}
        \ar[ld]^{m}
        \\
        & 
        H
        \ar[lu]^{\Delta}\ar[rr]_{{f\smashprod g}} 
        && 
        H
      }
    \end{gathered}
  \end{equation}
  where the map $\mathrm{cyclic}:H^{\otimes 3}\to H^{\otimes 3}$ is
  $x\otimes y\otimes z\mapsto y\otimes z\otimes x$. 
  The associativity of the Heisenberg product follows from the Hopf algebra axioms and its unit is the map $\iota\varepsilon$.

In explicit terms one has: 
\begin{equation}
    \label{D:smash-explicit}
(f\smashprod g)(h)
    =
    \sum 
    {f(h_1)}_2
    g\bigl( h_2 {f(h_1)}_1 \bigr).
\end{equation}
\end{definition}

We call $\End_f(H)$ the subspace of $\End(H)$ of finite rank linear homomorphisms, i.e. the image of the canonical inclusion $H^*\otimes H\mono\End(H)$. It is clear that the three above operations restrict to $\End_f(H)$ --observe that for the linear generators of 
 $\End_f(H)$ the Heisenberg product takes the following form:
 $\alpha|h \# \beta|\ell = \sum \alpha(h_1 \rightharpoonup \beta)|h_2\ell$, where $(\alpha|h)(h')=\alpha(h')h$ for $\alpha \in H^*,\,h,\,h'\in H$--.   It is also clear that if $H$ is finite dimensional $\End_f(H)=\End(H)$, is endowed with a coproduct given by the tensor products of the  coproducts in $H^*$ and $H$.  This coproduct, is compatible with the convolution product but not with the others. 
\begin{remark}
The Heisenberg product appears in the literature in different 
settings (see for example ~\cite{montgomery93}). Given a Hopf algebra $H$ and a 
$H$-module algebra $A$, the Heisenberg
product is defined as the operation on the space $A\otimes H$ given as:
\begin{equation}
  \label{E:smash-usual}
  (a\otimes h) \smashprod (b\otimes k)
  =
  \sum a(h_1\cdot b)\otimes h_2k.
\end{equation}
If $A=H^*$ and $H$ acts on $A$ by translations
then~\eqref{E:smash-usual} corresponds to~\eqref{D:smash-diagram} via
the canonical inclusion $H^*\otimes H\mono\End(H)$. Note that in the definition of the Heisenberg product in $\End(H)$ there are no restrictions about the dimensions.
\end{remark}

\subsection{The case of $K$--equivariant endomorphisms}

We need an equivariant version of the above construction. 

Assume that $K$ is another bialgebra and that 
 $H$ is a $K$--module
 bialgebra : $H$ is endowed with a left action of the algebra $K$ --if $k
 \in K$ on $h \in H$ the action is denoted as $(k,h)\mapsto k\cdot h: K \times H \rightarrow H$--. We have  that for all $k
 \in K$ and $h,\ell \in H$, 
$\Delta(k\cdot h)=  \sum k_1\cdot h_1 \otimes k_2 \cdot h_2$ and
 $k\cdot (h\ell)=\sum (k_1\cdot h)(k_2 \cdot \ell)$. 

The action of $K$ on $H$, induces a {\em right} action of $K$ on $H^*$
 via the formula: $\alpha \in H^*$, $h \in H$, $k \in K$, $(\alpha
 \leftharpoonup k)(h) = \alpha(k \cdot h)$. With respect to this
 action and if $H$ is finite dimensional, 
$H^*$ becomes a {\em right} $K$--module bialgebra. In general $H^*$ is
 only a $K$--module algebra.  

\begin{definition} Assume that $H$ is a
 $K$--module bialgebra for a certain Hopf algebra $K$. We 
define an action of $K$ on $\End(H)$ as follows: if $k
 \in K$ and $f \in {\End}(H)$, $(k\cdot f)(h)
 = \sum k_1\cdot f(S(k_2)h)$, --$S:K \rightarrow K$ denotes
 the antipode--. Explictly, the action on the generators of $\End_f(H)$ is the following: for $\alpha|h \in {\End_f}(H)$ and $k \in K$ : $k \cdot
 (\alpha|h)= \sum (\alpha \leftharpoonup Sk_2)| k_1 \cdot h$. 
\end{definition}

\begin{remark} \label{O:equivariant}\leavevmode
\begin{enumerate} 
\item The $K$-invariant elements
  for this action, i.e. the elements $f \in {\End}(H)$
such that for all $k \in K$, $k \cdot f = \varepsilon(k)
 f$, are the $K$--equivariant homomorhisms. They form a 
 vector subspace of $\End(H)$ denoted as
$\End_K(H)$. Clearly, the $K$--action is compatible with composition. 
\item In the case that $K$ is
 cocommutative, the $K$--action is also compatible with the convolution product. Then ${\End}(H)$ and ${\End_f}(H)$ are $K$--module algebras with convolution and ${\End}_{K}(H)$, ${\End}_{f,K}(H)$ are subalgebras.

\item In the case that $H$ is finite dimensional and $K$ is cocommutative, the coproduct of ${\End}(H)$ is compatible with the
action of $K$.  
Indeed, if $\alpha \in H^*$, $h \in H$, $k \in K$, we
have that $\Delta(k\cdot (\alpha|h))= \Delta ((\alpha \leftharpoonup Sk_2)|
k_1\cdot h)= \sum (\alpha \leftharpoonup Sk_2)_1|(k_1\cdot h)_1 \otimes 
(\alpha \leftharpoonup Sk_2)_2|(k_1\cdot h)_2= \sum (\alpha_1 \leftharpoonup
Sk_4)|(k_1\cdot h_1) \otimes (\alpha_2 \leftharpoonup Sk_3)|(k_2\cdot
h_2) = \sum (\alpha_1 \leftharpoonup
Sk_2)|(k_1\cdot h_1) \otimes (\alpha_2 \leftharpoonup Sk_4)|(k_3\cdot
h_2) = \sum k_1 \cdot (\alpha_1|h_1) \otimes k_2 \cdot (\alpha_2|h_2)$.  

\end{enumerate}
\end{remark} 
 




 

 \begin{lemma}\label{lemm:subalg1}\leavevmode
\begin{enumerate}\item In the situation above, if $K$ is
 cocommutative and $H$ is a $K$--module bialgebra, then $\End(H)$ endowed with the Heisenberg product is a $K$--module
 algebra and ${\End}_f(H)$ is a $K$--subalgebra.

\item Moreover, $\End_{K}(H)$ and
 $\End_{f,K}(H)$ are 
   $\smashprod$--subalgebras of $\End(H)$.
\end{enumerate} 
\end{lemma}
\begin{proof} We prove only the assertion concerning the Heisenberg product in
 ${\End}_f(H)$ as is the one we use in the applications. The rest of the proof is left to the reader. 
Consider $\alpha,\beta \in H^*$, $h,\ell \in H$ and $k \in K$. 

We first compute: 
\begin{equation}\label{E:equiv1} 
\begin{aligned}k\cdot(\alpha|h \# \beta|\ell)&=\sum k\cdot (\alpha(h_1
\rightharpoonup \beta) | h_2\ell)= \\ \sum (\alpha(h_1
\rightharpoonup \beta))\leftharpoonup Sk_2  | k_1\cdot(h_2\ell)&= 
\sum (\alpha(h_1 \rightharpoonup \beta))\leftharpoonup Sk_3  |
(k_1\cdot h_2)(k_2\cdot \ell) =\\ \sum (\alpha \leftharpoonup Sk_4) 
((h_1 \rightharpoonup \beta)\leftharpoonup Sk_3 ) &| 
(k_1\cdot h_2)(k_2\cdot \ell) ,
\end{aligned}
\end{equation}

\noindent next:
\begin{equation}\label{E:equiv2}
\begin{aligned}\sum k_1\cdot(\alpha|h) \# k_2\cdot (\beta|\ell)= \sum (\alpha \leftharpoonup Sk_2)|
 k_1\cdot h)
 \# (\beta \leftharpoonup Sk_4)|k_3\cdot \ell =\\
\sum (\alpha \leftharpoonup Sk_2)((k_1\cdot h)_1
\rightharpoonup (\beta \leftharpoonup Sk_4)) | (k_1\cdot h)_2 k_3\cdot \ell = \\
\sum (\alpha \leftharpoonup Sk_3)(k_1\cdot h_1
\rightharpoonup (\beta \leftharpoonup Sk_5)) | 
(k_2\cdot h_2)(k_4\cdot \ell)=\\  
\sum (\alpha \leftharpoonup Sk_3)(k_4\cdot h_1
\rightharpoonup (\beta \leftharpoonup Sk_5)) | (k_1\cdot h_2)(k_2\cdot \ell) \,.
\end{aligned}
\end{equation}

Now, the equality of \eqref{E:equiv1} and \eqref{E:equiv2}  
can be deduced  from the following calculation: 
take $\beta \in H^*$ , $h,r \in H$ and 
$k \in K$,

$\sum (k_1\cdot h
\rightharpoonup (\beta \leftharpoonup Sk_2))(r)= \sum (\beta \leftharpoonup
Sk_2)(rk_1\cdot h)=\sum \beta((Sk_2) \cdot (r (k_1 \cdot h)))= \sum \beta((Sk_3
\cdot r) (Sk_2 k_1 \cdot h))=\sum \beta ((Sk \cdot r )h) = ((h
\rightharpoonup \beta) \leftharpoonup Sk)(r)$.

\end{proof}

\subsection{The case of endomorphisms of graded Hopf algebras}

Assume that 
$H= \bigoplus_{n \geq 0} H_n$ is a 
graded connected bialgebra, i.e. for all $n,m \in \mathbb {N}$, $H_n H_m
\subset H_{n+m}$, $\Delta(H_n) \subseteq \bigoplus_{p+q=n} H_p
 \otimes H_q$ and $H_0=\field$ --conectivity condition--. It is well known that in this situation $H$ is a Hopf algebra, and that the antipode preserves the degree.  



\begin{definition}\label{D:endomorphisms} Consider the following chain of linear subspaces of $\End(H)$, where $\End_{\text{gr}}(H)$ is the subspace of the linear endomorphisms of $H$ that
 preserve the degree:
 $\End(H) \supseteq \End_{\text{gr}}(H)
 = \prod_n \End (H_n) \supseteq \bigoplus_n \End(H_n):= \operatorname {end}(H) \supseteq \bigoplus_n \End_f(H_n):= \operatorname {end}_f(H) $.  
\end{definition}

In the case that each $H_n$ is finite dimensional, $\operatorname {end}(H)= \operatorname {end}_f(H) = \bigoplus_n (H_n^* \otimes H_n)$, that can be endowed with  with a coproduct defined as below. 

\begin{definition}\label{defi:coprod-grad}
 In the situation above, take $\alpha|h \in
H^*_n \otimes H_n$, if $\Delta(\alpha)=\sum_{p+q=n} \alpha_p \otimes \alpha_q$
and $\Delta(h)=\sum_{r+s=n} h_r \otimes h_s$, with $\alpha_p \in H^*_p$ , 
$\alpha_q \in H^*_q$, $h_r \in H_r$, $h_s \in H_s$; then
$\Delta(\alpha|h)=\sum_{a+b=n} \alpha_a|h_a \otimes \alpha_b|h_b$. 
\end{definition}

\begin{remark} \label{R:compatigr} It is clear that the composition and convolution product defined in  $\End(H)$ restricts to the chain of subspaces considered above.  Moreover, in the case that the $H_n$ are finite dimensional, $\operatorname {end}(H)$ endowed with the convolution product and the above defined coproduct is a graded bialgebra.
\end{remark}

The behaviour of the Heisenberg product in the graded case is described in
the proposition that follows, that plays a central role in our constructions.

\begin{proposition}
  \label{P:interpolation}

\noindent(1)  The Heisenberg product of $\End(H)$
  restricts to $\gradedEnd(H)$.  Moreover, if $f\in\End(H_p)$ and
  $g\in\End(H_q)$ then
  \begin{equation}
    \label{E:smash-degree}
    f\smashprod g \in \bigoplus_{n=\max(p,q)}^{p+q} \End(H_n)
  \end{equation}
  and the top and bottom components of $f\smashprod g$ are
  \begin{equation}
    \label{E:smash-endomorphisms-extremes}
    (f\smashprod g)_{p+q}
    =
    f\externalend g
    \quad
    \text{and, if $p=q$,}
    \quad
    (f\smashprod g)_p=g\internalend f.
  \end{equation}

\noindent (2) In the case that $f= \alpha|k \in \End(H_p)$ and 
 $g = \beta|\ell \in \End(H_q)$, we have: 
\[\alpha|k \# \beta|\ell= \sum_{0 \leq n
   \leq \operatorname {min}(p,q)} \alpha(k_n \rightharpoonup \beta)|k_{p-n}\ell\],
   if $\Delta(k)=\sum_{n} k_{n} \otimes k_{p-n}$ where $k_{n} \in H_n$
   and $k_{p-n} \in H_{p-n}$.  Hence the Heisenberg product of $\End(H)$
  also restricts to $\gradedEnd_f(H)$.  
\end{proposition}

\begin{proof}
\noindent (1) Let $h\in H_n$.  The coproduct of $h$ is
  \begin{equation*}
    \Delta(h)
    =
    \sum_{a+b=n} h_a \otimes h_b,
  \end{equation*}
  with $h_a\in H_a$ and $h_b\in H_b$. Using the formula~\eqref{D:smash-explicit} we obtain:
  \begin{equation}
    \label{E:smash-endomorphisms-evaluated}
    (f\smashprod g)(h)
    =
    \sum_{a+b=n} 
    {f(h_a)}_2
    g\bigl( h_b {f(h_a)}_1 \bigr).
  \end{equation}
  Suppose that $f$ and $g$ belong to $\gradedEnd(H)$.  The computation
  of the degree of every term in the sum yields
  \begin{align*}
    \deg\Bigl(
    {f(h_a)}_2
    g\bigl( h_b {f(h_a)}_1 \bigr)
    \Bigr)
    &=
    \deg\bigl( {f(h_a)}_2 \bigr) 
    + 
    \deg\Bigl( g\bigl( h_b {f(h_a)}_1 \bigr) \Bigr)
    \\
    &=
    \deg\bigl( {f(h_a)}_2 \bigr) 
    +
    \deg\bigl( h_b {f(h_a)}_1 \bigr)
    \\
    &=
    \deg\bigl( {f(h_a)}_2 \bigr) 
    +
    \deg(h_b)
    +
    \deg\bigl( {f(h_a)}_1 \bigr) 
    \\
    &=
    \deg\bigl( f(h_a)\bigr) + \deg(h_b)
    \\
    &=
    a+b = n,
  \end{align*}
  proving that $f\smashprod g$ is in $\gradedEnd(H)$.

  We can refine the previous analysis as follows.  Assume that
  $f\in\End(H_p)$ and $g\in\End(H_q)$.  Then,
  Expression~\eqref{E:smash-endomorphisms-evaluated} is zero unless
  \begin{equation}
    \label{E:non-null-term-condition}
    a=p
    \quad\text{and}\quad
    b+\deg\bigl( {f(h_a)}_1 \bigr) = q. 
  \end{equation} 
  Adding these two equations we get that $n=a+b \le p+q$.  On the
  other hand, $p=a\le a+b=n$ and $q = b + \deg\bigl( {f(h_a)}_1 \bigr)
  \le b+a = n$, hence $\max(p,q) \le n$.  This
  proves~\eqref{E:smash-degree}.

  If we set $n=p+q$ in~\eqref{E:non-null-term-condition} then we get
  $\deg\bigl({f(h_a)}_1\bigr)=0$,
  and~\eqref{E:smash-endomorphisms-evaluated} reduces to the convolution
  diagram in~\eqref{D:composition-convolution-diagram}.  If we set
  $n=p=q$, then $\deg(h_b)=\deg\bigl({f(h_a)}_2\bigr)=0$,
  and~\eqref{E:smash-endomorphisms-evaluated} reduces to
  $g\bigl(f(h)\bigr)=(g\internalend f)(h)$, which is the composition product.

\medskip
\noindent (2)  The equality $\alpha|k \# \beta|\ell= \sum_{0 \leq n \leq
   \operatorname {min}(p,q)} \alpha(k_n \rightharpoonup \beta)|k_{p-n}\ell$ where
 $\Delta(k)= \sum k_n \otimes k_{p-n}$ 
 follows immediately from the explicit formul{\ae}~\eqref{E:smash-endomorphisms-evaluated} and from the
 considerations of (1) with  the corresponding bounds for the degrees.
\end{proof}

Thus and as expected, the Heisenberg product interpolates between the composition and
convolution products.  The analogous interpolation property at all
other non-commutative levels (permutations and non-commutative
symmetric functions) is a consequence of this general result.

\begin{remark}\label{R:equivariantinterpolation}\leavevmode
\begin{enumerate}
\item Assume that that $K$ is a commutative Hopf algebra and that $H$ is graded connected $K$--module Hopf algebra as above. Assume also that the action of $K$ preserves the grading. In this situation one can consider the chain of subspaces of $\End(H)$ that follows: 
\[
\begin{array}{c c c c c c c}
\End(H) & \supseteq & \End_{\text{gr}}(H) & \supseteq & \operatorname{end}(H) & \supseteq & \operatorname{end}_f(H) \\
\bigcup\hspace{-.09cm}\bf{|}& & \bigcup\hspace{-.09cm}\bf{|}& & \bigcup\hspace{-.09cm}\bf{|} & & \bigcup\hspace{-.09cm}\bf{|}\\
\End_K(H) & \supseteq & \End_{\text{gr},K}(H) & \supseteq & \operatorname{end}_K(H) & \supseteq & \operatorname{end}_{f,K}(H). 
\end{array}
\]
\item In this context is clear that all the subspacs considered above are closed under the composition, convolution and Heisenberg products. 
In particular the following holds: if $f\in\End_K(H_p)$ and
  $g\in\End_K(H_q)$ then
\begin{equation}
    \label{E:smash-degree--equivariant}
    f\smashprod g \in \bigoplus_{n=\max(p,q)}^{p+q} \End_K(H_n).
  \end{equation} 
\end{enumerate}
\end{remark}

            \section{The Heisenberg product of Garsia-Reutenauer endomorphisms}
\label{S:descents}

In this section we define certain distinguished subspace of endomorphisms of the Hopf algebra $H$, that we call the  \emph{Garsia-Reutenauer endomorphisms}. Then we show that the Heisenberg product in $\End(H)$ (Section~\ref{S:endomorphisms}) can be restricted to this special subspace.

These endomorphisms are characterized in terms of their action on products
of primitive elements of $H$. 

The motivation for the definition is that in the case that $H$ is the tensor algebra of a vector space, an important result of Garsia and Reutenauer --see \cite{garsia89}-- relates this subspace with the space 
of non-commutative symmetric functions via  Schur-Weyl duality (Lemma~\ref{L:schur-weyl-duality} and Theorem~\ref{T:sigma-GR}).

\begin{definition}
  \label{D:descents-hopf}
  Let $H$ be an arbitrary Hopf algebra.  If $h_1,\dots,h_n
  \in H$, define 
  \begin{equation*}
    G(h_1,\dots,h_n)
    =
    \linearspan{h_{\sigma(1)}\cdots h_{\sigma (n)} \mid
      \sigma \in \S_n},
  \end{equation*}
  or in other words, $G(h_1,\dots,h_n)$ is the subspace generated by the products
  of the form $h_{\sigma(1)} \cdots h_{\sigma (n)}$ for $\sigma\in\S_n$.
\end{definition}

For later use, we record the explicit expressions of the comultiplication in elements that are products of primitives $h_1\cdots ḥ_n$ . 

We consider the set of $(p,q)$-shuffles --that is denoted as $\mathrm{Sh}(p,q)$--.  A {\em $(p,q)$-shuffle} is a permutation $\xi\in S_{p+q}$ such that
\[\xi(1)<\cdots<\xi(p) \text{ \ and \ } \xi(p+1)<\cdots<\xi(p+q)\,.\]
\medskip 
 
The comultiplication of  $h_1\cdots ḥ_n$ is given as:
\begin{equation}\label{E:explicit-coproduct-primitives}
\Delta(h_1\cdots h_n)=\sum_{p+q=n}
\sum_{\xi\in\mathrm{Sh}(p,q)} h_{\xi(1)}\cdots 
h_{\xi(p)} \otimes  h_{\xi(p+1)}\cdots
h_{\xi(p+q)} \,.
\end{equation}
The following lemma lists some of the basic properties of the
subspaces $G(h_1,\dots,h_n)$.

\begin{lemma}
  \label{L:coproduct-primitives}
  For any $h_1,\dots,h_n \in H$ we have:
  \begin{enumerate}[(i)]
  \item If $a\in G(h_1,\dots,h_k)$ and $b\in G(h_{k+1},\dots,h_n)$,
    then $ab\in G(h_1,\dots h_n)$.
  \item If $a\in G(h_1,\dots,h_n)$ and $h_1,\dots,h_n\in\prim(H)$,
    then
    \begin{equation*}
      \Delta(a)
      =
      \sum_{\substack{k+\ell = n\\\xi\in\Sh(k,\ell)}}
      a^{(1)}_\xi\otimes a^{(2)}_\xi,
    \end{equation*}
    where $a^{(1)}_\xi\in G(h_{\xi(1)},\dots,h_{\xi(k)})$ and
    $a^{(2)}_\xi\in G(h_{\xi(k+1)},\dots,h_{\xi(n)})$.
  \end{enumerate}
\end{lemma}

\begin{definition}\label{D:Garsia-Reutenauer}
Let $H$ be an arbitrary Hopf algebra. The space of {\em Garsia-Reutenauer endomorphisms\/} of $H$ --denoted as $\Sigma(H)$-- is:
  \begin{equation*}
    \Sigma(H)
    =
    \bigl\{
    f\in\End(H) \bigm|
    f\bigl( G(h_1,\dots,h_n) \bigr) \subseteq G(h_1,\dots,h_n)
    \text{\ for all $h_1,\dots,h_n\in\prim(H)$}
    \bigr\}.
  \end{equation*}
\end{definition}

The subspace considered above, plays for $\End(H)$ the same role that the subspace of descents plays for the algebra of permutations. 
\begin{theorem}
  \label{T:sigma-hopf}
  If $H$ is a Hopf algebra, the space $\Sigma(H)$ of Garsia-Reutenauer endomorphisms
  is a subalgebra of $\End(H)$ with respect to the Heisenberg product.
\end{theorem}

\begin{proof}
  Given a primitive element $h$, we have $\iota\varepsilon(h)=0$, hence
  the unit of the Heisenberg product is in $\Sigma(H)$.

  Take two endomorphisms $f$ and $g$ in $\Sigma(H)$, and let
  $h_1,\dots,h_n\in \prim(H)$.  Then, we have by
  definition --see \eqref{E:explicit-coproduct-primitives}, ~\eqref{D:smash-diagram} and~~\eqref{D:smash-explicit}--: 
  \begin{equation}
    \label{E:smash-endomorphisms-in-primitives}
    (f\smashprod g)(h_1\cdots h_n)
    =
    \sum_{\substack{k+\ell=n\\\xi\in\Sh(k,\ell)}}
    \bigl( f(h_{\xi(1)}\cdots h_{\xi(k)}) \bigr)_2
    g\Bigl(
    h_{\xi(k+1)}\cdots h_{\xi(n)}
    \bigl( f(h_{\xi(1)}\cdots h_{\xi(k)}) \bigr)_1
    \Bigr).
  \end{equation}
  As $f(h_{\xi(1)}\cdots h_{\xi(k)})\in G(h_{\xi(1)}\cdots
  h_{\xi(k)})$, it follows from Lemma~\ref{L:coproduct-primitives}
  that
  \begin{equation*}
    \Delta\bigl( f(h_{\xi(1)}\cdots h_{\xi(k)}) \bigr)
    =
    \sum_{\substack{r+s=k\\\eta\in\Sh(r,s)}}
    a^{(1)}_\eta\otimes a^{(2)}_\eta,
  \end{equation*}
  with $a^{(1)}_\eta\in G(h_{\xi\eta(1)},\dots,h_{\xi\eta(r)})$ and
  $a^{(2)}_\eta\in G(h_{\xi\eta(r+1)},\dots,h_{\xi\eta(k)})$.  Hence,
  we rewrite~\eqref{E:smash-endomorphisms-in-primitives} as
  \begin{equation*}
    (f\smashprod g)(h_1\cdots h_n)
    =
    \sum_{\substack{k+\ell=n\\\xi\in\Sh(k,\ell)}}
    a^{(2)}_\eta\; g( h_{\xi(k+1)}\cdots h_{\xi(n)} \,a^{(1)}_\eta).
  \end{equation*}
  But the argument of $g$ belongs to $G(h_{\xi(k+1)},\dots,
  h_{\xi(n)},h_{\xi\eta(1)},\dots,h_{\xi\eta(r)})$.  Using that
  $g\in\Sigma(H)$ and using part $(1)$ of
  Lemma~\ref{L:coproduct-primitives} we obtain that
  \begin{multline*}
    (f\smashprod g)(h_1\cdots h_n) 
    \in
    G(h_{\xi\eta(r+1)},\dots,h_{\xi\eta(k)},h_{\xi(k+1)},\dots,
    h_{\xi(n)},h_{\xi\eta(1)},\dots,h_{\xi\eta(r)})
    \\
    \subseteq G(h_1,\dots, h_n),
  \end{multline*}
  proving that $f\smashprod g\in \Sigma(H)$.
\end{proof}
\begin{remark} It is easy to show that $\Sigma(H)$ is also a subalgebra of $\End(H)$ with respect to the composition and convolution products. 
\end{remark}

In the situation that $H$ is a graded connected bialgebra, we can produce an homogeneous and equivariant version of the above results. 
 
We define the following chain of subspaces of $\Sigma(H)$ :

\[\Sigma_{\text{gr}}(H) : = \Sigma(H) \cap \End_{\text{gr}}(H) \supseteq \sigma(H) := \Sigma(H) \cap \operatorname{end}(H),\]

\noindent
and in the case that $K$ is a commutative bialgebra that acts on 
$H$, by homogenous bialgebra endomorphisms --see Section
\ref{S:endomorphisms}-- we define:
\[\Sigma(H) \cap \End_K(H) = \Sigma_{K}(H)  \supseteq \Sigma_{\text{gr},K}(H)  \supseteq 
\sigma_K(H).\]
\noindent

As all the objects described above are defined as intersections, it is clear that they are closed under  composition, convolution and Heisenberg product.


            \section{The Heisenberg product of permutations}
\label{S:permutations}

In order to translate the \smash product from endomorphisms of Hopf
algebras to permutations we specialize the constructions of
Section~\ref{S:endomorphisms} and apply the methods related to the Schur-Weyl duality theorem --see \cite{howe95}--.

Let
\begin{equation*}
T(V)=\bigoplus_{n\geq 0} V^{\otimes n}
\end{equation*}
be the tensor algebra of a finite dimensional vector space $V$.  It is a graded connected
Hopf algebra with product defined by concatenation and with coproduct
uniquely determined by the condition that the algebra generators --the elements $v \in V$-- are primitive:
\begin{equation}
  \label{D:coproduct-tensorial-algebra}
\Delta:  
  v
  \mapsto 
  1\otimes v+v\otimes 1
  \quad \text{for $v\in V$}.
\end{equation}
As $v_1\otimes\dots\otimes v_n
= v_1\cdots v_n$, we omit the tensors when writing elements
of $T(V)$.




The general linear group $\GL(V)$ acts on $V$ and hence on each
$V^{\otimes n}$ diagonally.  Schur-Weyl duality --as presented for example in \cite{Etingov13} or \cite{howe95}-- guarantees that the only endomorphisms of $T(V)$
which commute with the action of $\GL(V)$ are (linear combinations of)
permutations. 

Let
\begin{equation*}
  \Sinfty = \bigoplus_{n\geq 0}\field \S_n
\end{equation*}
be the direct sum of all symmetric group algebras.  The product in
$\Sinfty$ is defined on permutations as the usual composition --denoted by 
$\sigma\internalperm\tau$ or $ \sigma \tau$--  when
$\sigma$ and $\tau$ belong to the same homogeneous component of
$\Sinfty$, and is 0 in any other case.  The identity in $\S_n$ is
denoted by $\Id_n$.

\begin{lemma}[Schur-Weyl duality]
  \label{L:schur-weyl-duality} 
  In the notations above, let $\Psi$ be the map
  \begin{equation*}
    \Psi:\Sinfty 
    \to
    \gradedEnd_{\GL(V)}\bigl(T(V)\bigr),
  \end{equation*}
  defined by sending $\sigma\in\S_n$ to the endomorphism
  $\Psi(\sigma)$ of $T(V)$, which in degree $n$ is given by the right
  action of~$\sigma$ on $V^{\otimes n}$:
  \begin{equation*}
    v_1\cdots v_n
    \xlongmapsto{\Psi(\sigma)}
    v_{\sigma(1)} \cdots v_{\sigma(n)}
  \end{equation*}
  and is $0$ in the other homogeneous components.  Then, $\Psi$ is an homogeneous isomorphism of vector spaces.
\end{lemma}
\begin{definition}\label{D:heisenberg_per} The \emph{Heisenberg product of permutations} $\smashprod: \Sinfty \times \Sinfty \rightarrow \Sinfty$ is defined by the commutativity of the diagram below:

\begin{equation*}
  \xymatrix@R=20pt@C=20pt{
    \Sinfty \times \Sinfty 
    \ar[rrr]^-{\Psi \times \Psi}\ar[d]_{\smashprod}
    &&&
    \gradedEnd_{\operatorname{GL}(V)}\big(T(V)\big)\times \gradedEnd_{\operatorname{GL}(V)}\big(T(V)\big)
    \ar[d]^{\smashprod}
    \\
    \Sinfty 
    \ar[rrr]_-{\Psi}
    &&&
   \gradedEnd_{\operatorname{GL}(V)}\big(T(V)\big)}
\end{equation*}
\end{definition}

Compare the above definition with the results by Malvenuto and Reutenauer in \cite{malvenuto95} where the authors deal with the convolution product in $\Sinfty$. The considerations of Section \ref{S:endomorphisms}, guarantee that the same methods can be applied in the situation treated in Definition \ref{D:heisenberg_per} for the Heisenberg product. 


This method, presented in \cite{malvenuto95} and used above, is important because it could be applied to
other dualities than Schur-Weyl, i.e. to centralizer algebras of other groups (or Hopf algebras) acting on a tensor algebra. 

It also can be applied to other products of
endomorphisms, a remarkable case being that of the {\em Drinfel'd} product,
which is studied in~\cite{moreira08}.

Next we exhibit an explicit formula for the \smash product of two
permutations, that at the lowest and highest degree, yield the usual formul\ae\/ for the composition and the Malvenuto--Reutenauer products respectively --see the table appearing in Figure \ref{F:table-of-names}--. 

We establish the following notation for the $(p,q)$ shuffle of maximal length:
\begin{equation*}
  \beta_{p,q} = \left(
  \begin{matrix}
    1 & 2 & \cdots & p & p+1 & p+2 & \cdots & p+q \\
    q+1 & q+2 & \cdots & q+p & 1 & 2 & \cdots & q
  \end{matrix} \right).
\end{equation*}

Notice in particular that $\beta_{p,q}=\beta_{q,p}^{-1}$.

\medskip


\begin{theorem}
\begin{enumerate}
 \item  Let $\sigma\in\S_p$ and $\tau\in\S_q$. Then, the \smash product in
  $\Sinfty$ can be expressed as
  \begin{equation}
    \label{T:smash-permutations}
    \sigma\smashprod \tau
    =
    \sum_{n=\max(p,q)}^{p+q}
    \sum_{\substack{\xi\in\Sh(p,n-p)\\ \eta\in\Sh(p+q-n,n-q)}}
    \xi
    \bigl( (\sigma\eta) \times \Id_{n-p} \bigr)
    \beta_{2n-p-q, p+q-n}
    (\Id_{n-q}\times\tau).
  \end{equation}
\item

When $n=p+q$:
\begin{equation*}
  (\sigma \smashprod \tau)_n = \sigma\externalperm\tau
  =
  \sum_{\xi\in\Sh(p,q)}\xi(\sigma\times\tau),
\end{equation*} 
\noindent
where $\sigma \times \tau \in S_{p+q}$ via the standard inclusion:   \[(\sigma\times \tau)(i)=\begin{cases} \sigma(i)

 & \text{ if }1\leq i\leq p\,,\\
 p+\tau(i-p) & \text{ if }p+1\leq i\leq p+q\,. \end{cases}\]

\noindent
When $n=p=q$: 
\begin{equation*}
  (\sigma \smashprod \tau)_n = \sigma \tau.
\end{equation*}   

\end{enumerate}
\end{theorem}
\begin{proof} \begin{enumerate}
\item 
  Using~\eqref{E:explicit-coproduct-primitives} and \eqref{E:smash-endomorphisms-in-primitives} for the endomorphisms $\Psi(\sigma)$ and
  $\Psi(\tau)$ induced by the permutations $\sigma\in\S_p$ and
  $\tau\in\S_q$, respectively, we obtain:
  \begin{multline*}
    \bigl(\Psi(\sigma) \smashprod \Psi(\tau) \bigr) 
    (v_1\cdots v_n)
    \\
    =
    \sum_{\substack{r+s=n\\\xi\in\Sh(r,s)}}
    \bigl( \Psi(\sigma)(v_{\xi(1)}\cdots v_{\xi(r)}) \bigr)_2
    \;\Psi(\tau)
    \Bigl(
    v_{\xi_{r+1}}\cdots v_{\xi_n}\,
    \bigl( \Psi(\sigma)( v_{\xi(1)}\cdots v_{\xi(r)} ) \bigr)_1 
    \Bigr).
  \end{multline*}
  The only non-zero terms occur when $r=p$, hence
  \begin{multline*}
    \bigl(\Psi(\sigma) \smashprod \Psi(\tau)\bigr)
    (v_1\cdots v_n)
    \\
    \begin{aligned}
      &=
      \sum_{\xi\in\Sh(p,n-p)}
      \bigl( \Psi(\sigma)(v_{\xi(1)}\cdots v_{\xi(p)}) \bigr)_2
      \;\Psi(\tau)\Bigl(
      v_{\xi(p+1)}\cdots v_{\xi(n)}\,
      \bigl( \Psi(\sigma)( v_{\xi(1)}\cdots v_{\xi(p)} ) \bigr)_1
      \Bigr)
      \\
      &=
      \sum_{\substack{\xi\in\Sh(p,n-p)\\ u+v=p \\ \eta\in\Sh(u,v)}}
      v_{\xi\sigma\eta(u+1)} \cdots v_{\xi\sigma\eta(p)}
      \Psi(\tau)(
      v_{\xi(p+1)}\cdots v_{\xi(n)}
      v_{\xi\sigma\eta(1)} \cdots v_{\xi\sigma\eta(u)}
      )
      \\
      &=
      \sum_{\substack{\xi\in\Sh(p,n-p)\\\eta\in\Sh(p+q-n,n-q)}}
      v_{\xi\sigma\eta(p+q-n+1)} \cdots v_{\xi\sigma\eta(p)}
      v_{\xi\tau(p+1)}\cdots v_{\xi\tau(n)}
      v_{\xi\sigma\eta\tau(1)} \cdots v_{\xi\sigma\eta\tau(p+q-n)}
      \\
      &=
      \sum_{\substack{\xi\in\Sh(p,n-p)\\\eta\in\Sh(p+q-n,n-q)}}
      \Psi\Bigl[
      \xi
      \bigl( (\sigma\eta) \times \Id_{n-p} \bigr)
      \beta_{2n-p-q, p+q-n}
      (\Id_{n-q}\times\tau)
      \Bigr](v_1\cdots v_n),
    \end{aligned}
  \end{multline*}
  which proves the first part of the theorem. 

\item This part follows directly. Observe that for the case $n=p+q$ 
 we obtain the product of permutations as defined by
Malvenuto-Reutenauer in ~\cite{malvenuto93,malvenuto95}.

\noindent
In the case $n=p=q$ and since the action of $S_n$ on $V^{\otimes n}$ is from the right, the composition of permutations corresponds to composition of endomorphisms in the opposite order --compare with the results of Proposition \ref{P:interpolation}--.
\end{enumerate}
\end{proof}

For example, writing the permutations in word format we obtain that:
{\renewcommand\minalignsep{0pt}%
\begin{align*}
  12 \smashprod 132
  ={}&\textcolor{blue}{132 + 231 + 321}
  \\
  &+\Id_4 + 1243 + 1324 + 2134 + 2143 + 2314
  \\
  &+ 3124 + 3142 + 3214 + 4123 + 4132 + 4213
  \\
  &\textcolor{red}{{} + 12354 + 13254 + 14253 + 15243 + 23154}
  \\
  &\textcolor{red}{{} + 24153 + 25143 + 34152 + 35142 + 45132}.
\end{align*}}

Notice that the red terms correspond to the Malvenuto--Reutenauer product. 

            \section{The \smash product of non-commutative symmetric functions}
\label{S:ncsymmetric}

The {\em descent set} of a permutation $\sigma\in\S_n$ is the subset
of $[n-1]$ defined~by
\begin{equation*}
  \Des(\sigma)
  =
  \bigl\{
  i\in[n-1] \mid \sigma(i)>\sigma(i+1)
  \bigr\}.
\end{equation*}
Given $J\subseteq [n-1]$, define $\B_J$ as the set of permutations
$\sigma\in\S_n$ with $\Des(\sigma)\subseteq J$, and consider the
following elements of $\field\S_n$:
\begin{equation}\label{E:def-X}
  X_J
  = 
  \sum_{\sigma\in\B_J}
  \sigma.
\end{equation}
The family of subsets of $[n-1]$ is in bijective correspondence with the set of compositions of $n$. Recally that $\alpha=(a_1,\cdots,a_r)$ is a composition of $n$ if all the $a_i$ are positive integres and $\sum_{i=1}^r a_i=n$, in this situation we write $\alpha=(a_1,\cdots,a_r) \models n$. 

The bijection between compositions and subsets is:
\begin{equation*}
  (a_1,a_2,\ldots,a_r)
  \longleftrightarrow 
  \{a_1,a_1+a_2,\ldots,a_1+\cdots+a_{r-1}\}.
\end{equation*}
For instance, if $n=9$, then $X_{(1,2,4,2)}=X_{\{1,3,7\}}$ and $X_{(2,4,2,1)}=X_{\{2,6,8\}}$. 

\begin{definition}\label{D:defidescents}
Let $\Sigma_n=\field\{X_\alpha: \alpha \models n\} \subseteq \field\S_n$, be the subspace \emph{linearly spanned} by $X_\alpha$ with $\alpha$ 
composition of $n$ and 
\begin{equation*}
  \Sigma
  =
  \bigoplus_{n\geq 0}\Sigma_n.
\end{equation*}
\end{definition}  
 An important result of Garsia and Reutenauer characterizes 
the elements of
$\Sinfty$ whose images by $\Psi$ --in the notations of Lemma \ref{L:schur-weyl-duality}-- belong to $\sigma_{\operatorname{GL}(V)}\big(T(V)\big)$.  in terms of their action on the tensor
algebra.  Recall the Definition~\ref{D:Garsia-Reutenauer} of $\sigma_{\operatorname{GL}(V)}\big(T(V)\big)$. 


\begin{theorem}[Garsia--Reutenauer, \cite{garsia89}]
  \label{T:sigma-GR}
  Let $\Psi$ be isomorphism defined in Lemma \ref{L:schur-weyl-duality}, then the following diagram commutes: 
\begin{equation*}
  \xymatrix@R=20pt@C=20pt{
    \Sigma
    \ar@{}[r]|-*[@]{\subseteq}
    \ar[d]_{\Psi|_{\Sigma}} 
    & \Sinfty \ar[d]^-{\Psi}\\
   \sigma_{\operatorname{GL}(V)}\big(T(V)\big)
   \ar@{}[r]|-*[@]{\subseteq} 
    &
    \gradedEnd_{\operatorname{GL}(V)}\big(T(V)\big),}
    \end{equation*}
   and the map $\Psi|_{\Sigma}$ is surjective.
\end{theorem}

  

A fundamental result of Solomon~\cite{solomon76} states that $\Sigma_n$ is a
subalgebra of the symmetric group algebra $\field\S_n$ with the
composition product. This is \emph{Solomon's descent algebra}.  It is also well-known that $\Sigma$ is closed under the
external product~\cite{gelfand95,gessel84,malvenuto95}; in fact,
\begin{equation}
  \label{E:con-X}
  X_{(a_1,\ldots,a_r)}\externalsym X_{(b_1,\ldots,b_s)}
  =
  X_{(a_1,\ldots,a_r,b_1,\ldots,b_s)}.
\end{equation}
The space $\Sigma$ with the external product is the algebra of 
{\em non-commutative symmetric functions}.  

The following theorem generalizes these two situations in view of
the interpolation property of the Heisenberg product of permutations.
 
\begin{theorem}\label{T:smash-sigma}
  The subspace $\Sigma \subseteq \Sinfty$ is closed under the Heisenberg product.
\end{theorem}

\begin{proof}
  This a direct result from Theorem~\ref{T:sigma-GR},
  Theorem~\ref{T:sigma-hopf}, and Schur-Weyl duality
  (Lemma~\ref{L:schur-weyl-duality}).
\end{proof}




The next theorem gives another version of the same result with an explicit description of the value of $X_\alpha \smashprod X_\beta$ for $\alpha \models p$ and $\beta \models q$.  The structure coefficients of $X_\alpha \smashprod X_\beta$ are expressed in terms of the matrices
$\M_{\alpha,\beta}^n$ defined in Subsection~\ref{SS:complete-homogeneous}. 

In~\cite{moreira08} appears another combinatorial proof which extends a proof of
Schocker in ~\cite{schocker04} for the composition product.

Moreover, another proof of Theorem~\ref{T:smash-sigma} can be obtained by
extending the Heisenberg product to the Coxeter complex of the symmetric
group (that is, the faces of the permutahedron).  This makes a
connection with recent work of Brown, Mahajan, Schocker, and others on
this aspect of the theory of descent
algebras~\cite{brown04,aguiar04a,schocker04}.

Recall the following definition: take
$\alpha=(a_1,\ldots,a_r)\composition p$,
$\beta=(b_1,\ldots,b_s)\composition q$ two compositions and let $n$ be an integer, $\max(p,q)\leq n\leq p+q$. Let $a_0:=n-p$, $b_0:=n-q$, and
 $\M_{\alpha,\beta}^n$ be the set of all integral 
$(s+1)\times(r+1)$--matrices with non negative entries $M=(m_{ij})_{
    \substack{0\leq i\leq s\\\,0\leq j\leq r}}$ such that: the sequence of column sums  is $(a_0,a_1,\ldots,a_r)$, the sequence of row sums  is $(b_0,b_1,\ldots,b_s)$ and the first entry is $m_{00}=0$.
To visualize these conditions we write the diagram:
\begin{equation*}
    \begin{smasharray}{cccc!{\putbar}c}
      0 & m_{01} & \cdots & m_{0r} & n-q \\ 
      m_{10} & m_{11} & \cdots & m_{1r} & b_1 \\
      \vdots & \vdots & \ddots & \vdots & \vdots \\
      m_{s0} & m_{s1} & \cdots & m_{sr} & b_s \\ 
      \cline{1-4} 
      n-p & a_1  & \cdots &  a_r \omitbar
    \end{smasharray}
\end{equation*}

\begin{theorem}
  \label{T:smash-X}
  Let $\alpha\composition p$ and $\beta\composition q$ be two
  compositions.  Then
  \begin{equation}\label{E:smash-X}
    X_\alpha\smashprod X_\beta 
    = 
    \sum_{n=\max(p,q)}^{p+q}
    \sum_{M\in\M_{\alpha,\beta}^n} X_{c(M)}
  \end{equation}
  where $c(M)$ is the composition whose parts are the non-zero entries of $M$, read
  from left to right and from top to bottom.
\end{theorem}

Observe that even though this formula is similar to the one we had for
symmetric functions~\eqref{E:smash-complete-homogeneous}, the
occurrence of the compositions as indices of the basis makes the
connection between the Heisenberg product and the external and
Solomon products considerably subtler than in the commutatative
context.  In particular, a formula like~\eqref{E:zelevinski-identity}
does not longer hold. Moreover, the above expression will allow us to make the connection with the Heisenberg product of the representations of the symmetric group. This point is taken up in Section~\ref{S:nc-to-commutative}.

By the interpolation property of the Heisenberg product,
Theorem~\ref{T:smash-X} contains as special cases rules for the
product in Solomon's descent algebra and for the external product
of two basis elements of $\Sigma$.  One readily verifies that the
former is precisely the well-known rule of Garsia, Remmel, Reutenauer,
and Solomon as given in~\cite[Proposition 1.1]{garsia89}, while the latter
is the one appearing in~\eqref{E:con-X}.

As an example we have the following formula
\begin{equation}
  \label{E:smash-X_1-to-the-p}
  X_{(1^u)} \smashprod X_{(1^v)} 
  = 
  \sum_{n=\max(u,v)}^{u+v} 
  \binom{u}{n-v} \binom{v}{n-u} (u+v-n)! \, X_{(1^n)},
\end{equation}  
where $(1^n)$ is the composition of $n$ with $n$ parts equal to $1$.
The coefficients arise by inspection of all the possible ways to fill out the entries of the matrix $M$ in a diagram as below with row and column sums as prescribed. 
\begin{equation*}
  \begin{smasharray}{cccc!{\putbar}c}
    0 & ? & \cdots & ? & n-v \\ 
    ? & ? & \cdots & ? & 1 \\
    \vdots & \vdots & \ddots & \vdots & \vdots \\
    ? & ? & \cdots & ? & 1 \\
    \cline{1-4} 
    n-u & 1  & \cdots & 1 \omitbar \\
  \end{smasharray}
\end{equation*}

Similarly, one verifies that:
\begin{equation*}
  X_{(1)}^{\#(n)} = \sum_{k=1}^n S(n,k) X_{(1^k)},
\end{equation*}
where the $S(n,k)$ are the Stirling numbers of the second kind.

\begin{proof}[Proof of Theorem~\ref{T:smash-X}]
  Let us take a fixed integer $n$ between $\max(p,q)$ and $p+q$.  To
  the compositions $\alpha=(a_1,\ldots,a_r)$ and
  $\beta=(b_1,\ldots,b_s)$ we associate the following sets:
  \begin{align*}
    E^{n}_0 &= [p+1,n], 
    & 
    F^{n}_0 &= [1, n-q], \\
    E^{n}_1 &= [1,a_1], 
    &
    F^{n}_1 &= n-q + [1, b_1], \\
    E^{n}_2 &= [a_1+1, a_1+a_2], 
    &
    F^n_2 &= n-q + [b_1+1, b_1+b_2], \\
    &\centerinequal{\vdots} 
    & 
    &\centerinequal{\vdots} \\
    E^n_{r} &= [a_1+\cdots+a_{r-1}+1, p], 
    &
    F^n_{s} &= n-q + [b_1+\cdots+b_{s-1}+1, q].
  \end{align*}
  Observe that the family of intervals 
  $\{E^n_j\}_{j\in \{0,\ldots , r\}}$
  and $\{F^n_i\}_{i\in \{0,\ldots , s\}}$ are partitions of
  $[1,n]$.  It is also clear that $\sigma\in\B_\alpha$ if and only
  if $\sigma\times\Id_{n-p}$ is increasing in $E^n_j$ for all $j\in
  \{0,\ldots ,r\}$.  Similarly, $\tau\in\B_\beta$ if and only if
  $\Id_{n-q}\times\tau$ is increasing in $F^n_i$ for all $i\in
  \{0,\ldots ,s\}$.  Observe, also, that $\card{E^n_j}$ is the $j$-th
  coordinate of the pseudo-composition $(n-p,a_1,\ldots,a_r)$, and
  $\card{F^n_i}$ is the $i$-th coordinate of $(n-q,b_1,\ldots,b_s)$.

  Given $\eta \in \Sh(p+q-n,n-q)$ and $\tau \in \B_\beta$, call
  $\varphi_{\eta,\tau}=(\eta\times\Id_{n-p})\beta_0(\Id_{n-q}\times\tau)$
  and define the matrix
  \begin{equation*} 
    M_{\eta,\tau} 
    = 
    \bigl\{ 
    \card(F^n_i \cap \varphi_{\eta,\tau}^{-1}E^n_j)
    \bigr\}_{\substack{0\leq i \leq s,\\ 0 \leq j \leq r}}
  \end{equation*}
  where we have abbreviated $\beta_{2n-p-q,p+q-n} = \beta_{0}$.  In
  this situation $M_{\eta,\tau} \in \M_{\alpha,\beta}^n$.  Indeed, if
  we call $m_{ij}= \card(F^n_i \cap \varphi_{\eta,\tau}^{-1}E^n_j)$,
  for $i=j=0$ we have that
  \begin{equation*}
    \varphi_{\eta,\tau}[1,n-q]
    =
    \eta[p+q-n+1,p] \subseteq [1,p],
  \end{equation*}
  which shows that the intersection $F^n_0 \cap
  \varphi_{\eta,\tau}^{-1}E^n_0$ is empty, and then $m_{00}=0$.  The
  sum $m_{0j}+\cdots+m_{sj}$ equals the number of elements of $E^n_j$,
  which is, as noted before, the $j$-th entry of the composition
  $(n-p,a_1,\ldots,a_r)$.  The same argument applies to the sum of the
  rows.  In this manner, sending $\tau \mapsto M_{\eta,\tau}$ we
  define a map $\B_\beta \to \M_{\alpha,\beta}^n$.  Take $M =
  \{m_{ij}\} \in \M^n_{\alpha,\beta}$ and let $\B_\beta^{\eta,n}(M)$
  the corresponding fiber of this map:
  \begin{equation*}
    \B_\beta^{\eta,n}(M) 
    = 
    \bigl\{ 
    \tau\in\B_\beta 
    \bigm|
    \card(F^n_i \cap \varphi_{\eta,\tau}^{-1}E^n_j ) = m_{ij}
    \text{\ for all $j\in\{0,\dots,r\}$, $i\in\{0,\ldots ,s\}$}
    \bigr\}.
  \end{equation*}
  Therefore, we have a partition of $\B_\beta = \bigcup_{M \in
    \M^n_{\alpha,\beta}} \B_\beta^{\eta,n}(M)$.

  For $\xi\in\Sh(p,n-p)$ and $\eta\in\Sh(p+q-n,n-q)$, let us denote
  $
  g^n_{\xi,\eta}(\sigma,\tau) 
  = 
  \xi\bigl(
  (\sigma\eta)\times\Id_{n-p}
  \bigr) 
  \beta_0
  (\Id_{n-q}\times\tau)
  $,
  the $n$-term in the sum~\eqref{T:smash-permutations}.  The function
  $g^n_{\xi,\eta}$ is bilinear, and we can write
  \begin{equation*}
    X_\alpha\smashprod X_\beta 
    = 
    \sum_{n} \sum_{\xi,\eta}
    g^n_{\xi,\eta} (X_\alpha, X_\beta).
  \end{equation*}

  From now on as $n$ is fixed we will omit it in the notations of the
  sets and the functions.  Next we show that
  \begin{equation*}
    \sum_{\xi,\eta}
    g_{\xi,\eta}(X_\alpha, X_\beta) 
    =
    \sum_{M \in \M_{\alpha,\beta}} X_{c(M)}. 
  \end{equation*}
  For this, we write
  \begin{align} 
    \label{E:sums_of_g}
    \sum_{\xi,\eta} g_{\xi,\eta}(X_\alpha, X_\beta) 
    &=
    \sum_{\xi,\eta} g_{\xi,\eta}
    \biggl(
    \sum_{\sigma\in\B_\alpha} \sigma,
    \sum_{M\in\M_{\alpha,\beta}}
    \sum_{\tau\in\B_\beta^\eta(M)} \tau 
    \biggr) \nonumber 
    \\
    &= 
    \sum_{M\in\M_{\alpha,\beta}}
    \sum_{\xi,\eta}
    \sum_{\sigma\in\B_\alpha}
    \sum_{\tau\in\B_\beta^\eta(M)} 
    g_{\xi,\eta}(\sigma,\tau).
  \end{align}
  If we denote by $S_{\alpha,\beta}(M)$ the set of elements
  $(\xi,\eta,\sigma,\tau)$ such that $\xi\in\Sh(p,n-p)$,
  $\eta\in\Sh(p+q-n,n-q)$, $\sigma\in\B_\alpha$ and
  $\tau\in\B_\alpha^\eta(M)$; then the map
  $\psi:S_{\alpha,\beta}(M)\to \B_{c(M)}$ given by
  $\psi(\xi,\eta,\sigma,\tau)=g_{\xi,\eta}(\sigma, \tau)$ is a
  bijection.  We prove this fact in Lemma~\ref{L:bijection}.  In this
  situation, if we group together the last three sums
  of~\eqref{E:sums_of_g} we obtain
  \begin{equation*}
    \sum_{\xi,\eta} g_{\xi,\eta}(X_\alpha, X_\beta) =
    \sum_{M\in\M_{\alpha,\beta}} X_{c(M)},
  \end{equation*}
  which concludes the proof of the theorem.
\end{proof}

In the following two lemmas we assume the notations of the previous
theorem.  Their proofs, being rather technical are presented in the Appendix. See:~\ref{P:interval}
and~\ref{P:bijection}.

\begin{lemma} 
  \label{L:interval}
  For $\eta\in\Sh(p+q-n,n-q)$, $\tau\in\B_\beta$ and for all
  $i=0,\ldots,s$ and $j=0,\ldots ,r$, the sets
  \begin{equation*}
    F_i \cap \varphi_{\eta,\tau}^{-1}E_j
  \end{equation*}
  are disjoint intervals.  Moreover, in each of these intervals the
  function $\varphi_{\eta,\tau}$ is increasing and has image either
  contained in $[1,p]$ or contained in $[p+1,n]$.
\end{lemma}

\begin{lemma} 
  \label{L:bijection}
  For $M\in\M_{\alpha,\beta}$, the map $\psi:S_{\alpha,\beta}(M) \to
  \B_{c(M)}$, which sends $(\xi,\eta,\sigma,\tau)$ into
  $g_{\xi,\eta}(\sigma,\tau)$, is a bijection.
\end{lemma}


            \section{From non-commutative to commutative symmetric functions}
\label{S:nc-to-commutative}

In the previous four sections of \emph {Part 2: Non--commutative context}; we constructed the following commutative
diagram of algebras --endowed with their respective Heisenberg products-- 
(omitting the part that includes the space $\Lambda$):
\begin{equation*}
  \xymatrix@R=20pt@C=20pt{\Lambda&
    \Sigma
    \ar@{..>}[l]_{\pi}\ar@{}[r]|-*[@]{\subseteq}
    \ar[d]_{\Psi|_{\Sigma}} 
    & \Sinfty \ar[d]^-{\Psi}&\\
   &\sigma_{\operatorname{GL}(V)}\big(T(V)\big)
   \ar@{}[r]|-*[@]{\subseteq} 
    &
    \gradedEnd_{\operatorname{GL}(V)}\big(T(V)\big)\ar@{}[r]|-*[@]{\subseteq}&\gradedEnd\bigl(T(V)\bigr).}
    \end{equation*}

In the part of the above diagram that excludes $\Lambda$, all the Heisenberg products are induced by the one defined in the larger space $\gradedEnd\bigl(T(V)\bigr)$. Recall also that the vertical isomorphisms are due to Schur--Weyl duality and to the results of \cite{garsia89} as mentioned in Section \ref{S:ncsymmetric}, Theorem \ref{T:sigma-GR}.

We want to incorporate into this diagram the space of symmetric functions  $\sym$ equipped also with the Heisenberg product, this is expressed in the dotted arrow that is to be defined.

For each $n\ge 0$ we define the linear map $\pi_n:\Sigma_n\to\Lambda_n$ by
its values on the basis $\{X_\alpha: \text{$\alpha$ composition of}\,\, n\}$ as: 
\begin{equation*}
  \pi_n(X_\alpha) = h_{\widetilde\alpha},
\end{equation*}
where $\widetilde\alpha$ is the
partition of $n$ obtained by reordering the entries of the composition $\alpha$.
Let us denote by $\pi:\Sigma\epi\Lambda$ the map induced in the direct
sums.


It is well known that $\pi$ is a morphism if we endow $\Sigma$  with the Solomon product and $\Lambda$ with the internal product, it is also a morphism when we endow both spaces with the external products.  

The theorem that follows generalizes these compatibilities by proving that 
the map $\pi:\Sigma\epi\sym$ is a morphism with respect to 
the Heisenberg products in $\Lambda$ and
$\Sigma$, as constructed in Sections ~\ref{S:symmetric} and~\ref{S:ncsymmetric}, respectively.

\begin{theorem}
  For any pair of compositions $\alpha$ and $\beta$:
  \begin{equation*}
    \pi(X_\alpha\smashprod X_\beta)
    =
    h_{\widetilde\alpha} \smashprod h_{\widetilde\beta}.
  \end{equation*}
\end{theorem}

\begin{proof}
  Comparing the explicit formul\ae --see equations~\eqref{E:smash-complete-homogeneous}
  and~\eqref{E:smash-X}--  it is clear that it is enough to construct for each $n$ a bijection
  $\psi:\M_{\alpha,\beta}^n\to\M_{\widetilde\alpha,\widetilde\beta}^n$, such that
  for all $M\in\M_{\alpha,\beta}^n$:

\begin{equation}
    \label{E:reordering-bijection}
    p\bigl( \psi(M) \bigr)
    =
    \widetilde{c(M)}
  \end{equation}

  Let $\sigma$ and $\tau$ be two permutations which reorder 
 into partitions the
  compositions $\alpha=(a_1,\dots,a_r)$ and $\beta=(b_1,\dots,b_s)$,
  respectively,  i.e.:
  \begin{equation*}
    \widetilde\alpha = (a_{\sigma(1)},\ldots,a_{\sigma(r)}), 
    \quad
    \widetilde\beta = (b_{\tau(1)},\ldots,b_{\tau(r)}).
  \end{equation*}
  Define $\psi(M)$ as the matrix obtained from $M$ by permuting its
  columns with the permutation $\Id_1\times\sigma$ and its rows with
  the permutation $\Id_1\times\tau$.  Clearly $\psi(M)$ is in 
  $\M_{\widetilde\alpha,\widetilde\beta}^n$ and the map $\psi$ is  a bijection.
  
Moreover, since the (unordered) entries of $M$ and $\psi(M)$ are the same, we
  get Equation~\eqref{E:reordering-bijection}.
\end{proof}


            \section{Compatibility of the coproduct with the Heisenberg product} 
\label{S:hopf}
In various of the spaces depicted in the table in Figure \ref{F:spaces}, coproducts can be introduced, and frequently they are compatible with the products we are considering. For brevity we concentrate in the consideration of the compatibility of the coproducts with the Heisenberg product at the level of permutations, of non commutative and of commutative symmetric functions. We start by considering a coproduct that is called the convolution (or external) coproduct.

We start by defining this coproduct at the level of  $\Sinfty$. 

Let $n$ be a positive integer and decompose it as $n=p+q$. 
Then, as ${\mathrm {Sh}}^{-1}(p,q)$ is a set of 
representatives for the left cosets of
${S}_p\times {S}_q \subseteq {S}_n$, given $\sigma\in {S}_{n}$ there is a unique triple
$(\xi,\sigma_p,\sigma'_q)$ such that: $\xi\in{\mathrm {Sh}}(p,q)$,
$\sigma_p\in{S}_p$, $\sigma'_q \in{S }_q$, and
\begin{equation}\label{E:parabolic}
 \sigma=(\sigma_p\times \sigma'_q )\xi^{-1}\,.
\end{equation}

The coproduct $\Delta:\Sinfty\to \Sinfty\otimes
\Sinfty$ is defined on $\sigma\in{S}_n$ as
\begin{equation} \label{E:coprodper}
\Delta(\sigma)=\sum_{p=0}^n \sigma_p \otimes 
\sigma'_q,
\end{equation}where $\sigma_p$ and $\sigma'_q$ are as
in~\eqref{E:parabolic}.
For instance : $\Delta(52413)=(\,\,)\otimes (52413)+(1)\otimes (4132)+(21)\otimes (321)+(231)\otimes
(21)+(2413)\otimes (1)+(52413)\otimes (\,\,)
$.

If follows directly from the formula \eqref{E:coprodper} above (see \cite{malvenuto93}) that:
\begin{equation}
  \label{D:coproduct-sigma}
  \Delta(X_{(a_1,\ldots,a_r)})
  =
  \sum_{\substack{b_i+c_i=a_i\\0\le b_i,c_i \le a_i}}
  X_{\pseudopartition{(b_1,\dots,b_r)}}\otimes X_{\pseudopartition{(c_1,\dots,c_r)}},
\end{equation}
where for a pseudopartition $\alpha$, $\alpha\pseudopartition{}$  indicates that parts equal to zero have been omitted.  

It is then clear that the comultiplication $\Delta$ can be restricted to the space of descents: i.e. that \[\Delta(\Sigma) \subseteq \Sigma \otimes  \Sigma.\] 

\begin{remark}
  \label{R:coproduct-in-S}
  
In \cite[Th\`eor\'eme 5.3, Remarque 5.15]{malvenuto93} it is proved that
equipped with the convolution product and the above coproduct,
$\Sinfty$ becomes a graded connected Hopf
algebra and it is also shown that $\Delta$ 
is not compatible with the composition of permutations. Also in~\cite{aguiar05}  a more recent and detalied study of $(\Sinfty, \externalperm, \Delta)$ is presented. 

Taking into account that the Heisenberg product, interpolates between the convolution (or Malvenuto--Reutenauer) and the composition product at the level of the  permutations, we cannot expect it to be compatible with the coproduct considered above. 

These operations are better behaved if we restrict our attention to the non commutative symmetric functions.
\end{remark}

Next we prove that $\Delta$ is compatible with the Heisenberg product in
$\Sigma$, and in particular this implies that for descents the composition product \emph{is} compatible with the comultiplication.

\begin{theorem}
  \label{T:sigma-is-smash-hopf}
  The space $(\Sigma,\smashprod,\Delta)$ is a cocommutative Hopf
  algebra.
\end{theorem}

\begin{proof}
  It is enough to prove the Heisenberg--multiplicativity of $\Delta$ on elements of the form $X_\alpha$ with $\alpha$ a composition of $p$ for different $p$'s.    

Let $\alpha$ and $\beta$ be compositions
  of $p$ and $q$, respectively.  We use Formula~\eqref{E:smash-X} to
  compute
  \begin{equation}
    \label{E:coproduct-of-smash}
    \Delta(X_\alpha\smashprod X_\beta)
    =
    \sum_{n}
    \sum_{M\in\M_{\alpha,\beta}^n}
    \Delta( X_{c(M)} )
    =
    \sum_{n=\operatorname{max}(p,q)}^{p+q}
    \sum_{M\in\M_{\alpha,\beta}^n}
    \sum_{\gamma+\gamma'=c(M)}
    X_\gamma\otimes X_{\gamma'}.
  \end{equation}
  On the other hand,
  \begin{equation}
    \label{E:smash-of-coproducts}
    \begin{aligned}
      \Delta( X_\alpha ) \smashprod \Delta( X_\beta )
      &=
      \biggl(
      \sum_{\gamma + \gamma'=\alpha} X_\gamma\otimes X_{\gamma'}
      \biggr) 
      \smashprod
      \biggl(
      \sum_{\delta + \delta'=\beta} X_\delta\otimes X_{\delta'}
      \biggr)
      \\
      &=
      \sum_{\substack{\gamma+\gamma'=\alpha\\\delta+\delta'=\beta}}
      (X_\gamma\smashprod X_\delta) \otimes
      (X_{\gamma'}\smashprod X_{\delta'})
      \\
      &=
      \sum_{\substack{\gamma+\gamma'=\alpha\\\delta+\delta'=\beta}}
      \sum_{n,n'}
      \sum_{\substack{M\in\M^n_{\delta,\gamma}\\M'\in\M^{n'}_{\delta',\gamma'}}}
      X_{c(M)}\otimes X_{c(M')}.
    \end{aligned}
  \end{equation}
  We show that the sums~\eqref{E:coproduct-of-smash}
  and~\eqref{E:smash-of-coproducts} are the same as follows: take an
  octuple of indices corresponding to the
  sum~\eqref{E:smash-of-coproducts}:
  $(\gamma,\gamma',\delta,\delta',n,n',M,M')$ and construct the
  quadruple $\bigl(n+n', M+M', c(M), c(M')\bigr)$.  Denote by
  $\col{M}$ ($\row{M}$) the vector whose entries are the sum of the columns (rows) of the
  matrix $M$.  Since
  \begin{align*}
    \col{M+M'} 
    = 
    \col{M}+\col{M'} 
    =
    \bigl(n-|\gamma|\bigr)\gamma + \bigl(n'-|\gamma'|\bigr)\gamma'
    = 
    ( n+n' - p)\alpha,
  \end{align*}
  where $|\zeta|$ is the sum of the parts of a composition $\zeta$,
  and similarly with $\row{M+M'}=(n+n'-q)\beta$, we see that $M+M'\in
  \M^n_{\alpha,\beta}$.  As $c(M)+c(M')=c(M+M')$,  if we set
  \begin{equation*}
    (\widetilde n, \widetilde M, \widetilde\gamma, \widetilde{\gamma'})
    =\bigl(n+n', M+M', c(M), c(M')\bigr)
  \end{equation*}
  it is clear that $(\widetilde n,\widetilde M,\widetilde\gamma,
  \widetilde{\gamma'})$ is a quadruple of indices appearing in the
  sum~\eqref{E:coproduct-of-smash} and that the corresponding summands
  of~\eqref{E:coproduct-of-smash} and~\eqref{E:smash-of-coproducts}
  are the same.

  Moreover, it is clear that the above correspondence between the
  indices of the sums is bijective.
\end{proof}

Consider now the space of symmetric functions $\Lambda$ and the proyection $\pi: \Sigma \rightarrow \Lambda$. In Subsections \ref{SS:complete-homogeneous} and \ref{SS:powersums}, equations \eqref{E:coproductoncompletehom} and \eqref{E:smash-power-sums} we defined a coproduct on $\Lambda$ similar to the one defined above and dual to the external product. We then gave   its expression on the natural basis of complete homogeneous functions and power sums. 

In particular we have that --see the notations of Subsection \ref{SS:complete-homogeneous} and \ref{SS:powersums}--: $ \Delta(h_a)=\sum_{i+j=a}h_i \otimes h_j$ and $\Delta(p_n) = 1\otimes p_n+p_n\otimes 1$.

Being $\pi(X_\alpha)=h_\alpha$, and using the first of the above equalities and Theorem~\ref{T:sigma-is-smash-hopf},  we conclude the following result.

\begin{corollary}
  The space of symmetric functions $\sym$ equipped with the operations 
$(\smashprod,\Delta)$ is a
  cocommutative Hopf algebra and the map $\pi:\Sigma \rightarrow \Lambda$  is a morphism of Hopf algebras.
\end{corollary}
\begin{proof} We have that: 
\[\Delta(\pi(X_\alpha))=\Delta(h_\alpha)=\sum h_{\alpha_1} \otimes h_{\alpha_2}=\sum \pi(X_{\alpha_1}) \otimes \pi(X_{\alpha_2})=(\pi \otimes \pi)(\Delta(X_\alpha)).\]

Since $\pi$ is Heisenberg multiplicative, we conclude that the coproduct and the Heisenberg product are compatible in the space $\Sigma$. The compatibility of the coproduct and the Heisenberg product in $\Sigma$ induces the compatibility in $\Lambda$.
\end{proof}

This theorem generalizes known results on $(\sym,\externalsym,\Delta)$ and $(\sym,\internalsym,\Delta)$ --see for example \cite{geissinger77}--.



            \section{Isomorphisms between the Heisenberg, convolution, and composition
  products}
\label{S:isomorphisms}

We have mentioned until now, two kinds of relations between the three products we have been dealing with: one is the interpolation connection and the other the formula appearing in Lemma \ref{L:zelevinski-identity}.

In the spaces of symmetric functions there are further relations between the Heisenberg, internal and external products. 

First we show that the
external and Heisenberg products are isomorphic (in the commutative and non--commutative situations), but the isomorphism is
not degree-preserving. 

Similarly we prove that the Heisenberg and internal products are isomorphic in the commutative context --degrees not preserved--, but the isomorphism is only valid in the completion of the space.   

\begin{theorem}
  \label{T:iso-sigma}
  The map $\psi:(\Sigma,\externalsym,\Delta)\to
  (\Sigma,\smashprod,\Delta)$ given by
  \begin{equation}
    \label{E:iso-sigma}
    \psi(X_{(a_1,\ldots,a_r)})
    =
    X_{(a_1)}\smashprod\cdots\smashprod X_{(a_r)}
  \end{equation}
  is an isomorphism of Hopf algebras (which does not preserve the gradings).
\end{theorem}

\begin{proof}
  Since the Heisenberg product has the external product as the only term
  in the upper degree, the matrix of the linear map $\psi$ in the
  basis of the $X_\alpha$'s is triangular with 1 in the diagonal.  Hence $\psi$ 
  is invertible and it is multiplicative because the
  external product in the basis of the $X_\alpha$'s is the concatenation
  of the compositions.

  We finish by proving that $\psi$ is comultiplicative:
  \begin{equation}
    \label{E:psi-comultiplicative}
    \Delta\bigl(\psi(X_\alpha)\smashprod \psi(X_\beta)\bigr) 
    =
    (\psi\otimes\psi)\Delta(X_\alpha\externalsym X_\beta).
  \end{equation}
  Clearly, it is
  enough to prove~\eqref{E:psi-comultiplicative} on the algebra
  generators $X_{(a)}$ for all non-negative integers $a$.  For the right hand
  side of~\eqref{E:psi-comultiplicative} we have:
  \begin{equation*}
    \Delta(X_{(a_1)}\externalsym X_{(a_2)})
    =
    \Delta(X_{(a_1,a_2)})
    =
    \sum_{\substack{a+b=a_1\\a'+b'=a_2}}
    X_{\pseudopartition{(a,a')}}\otimes X_{\pseudopartition{(b,b')}}.
  \end{equation*}
  Applying the map $\psi\otimes\psi$ and using
  formula~\eqref{E:smash-X} to compute $\psi(X_{(a,a')}) =
  X_{(a)}\smashprod X_{(a')}$ and $\psi(X_{(b,b')}) =
  X_{(b)}\smashprod X_{(b')}$ (note that we assume $\psi(X_{(0)})$ to
  be the identity) we get
  \begin{equation}
    \label{E:psi-comultiplicative-RHS}
    (\psi\otimes\psi)\Delta(X_{\pseudopartition{(a_1,a_2)}})
    =
    \sum_{\substack{a+b=a_1\\a'+b'=a_2}}
    \sum_{n,m}
    X_{\pseudopartition{(n-a',n-a,a+a'-n)}}\otimes X_{\pseudopartition{(m-b',m-b,b+b'-m)}}.
  \end{equation}
  On the other hand, taking into account that
  $\psi(X_{\alpha})=X_\alpha$ for partitions with only one part, the
    left hand side of~\eqref{E:psi-comultiplicative} is:
  \begin{equation}
    \label{E:psi-comultiplicative-LHS}
    \Delta\bigl(X_{(a_1)}\smashprod X_{(a_2)}\bigr)
    =
    \sum_k
    \sum_{
      \substack{c_1+c_1'=k-a_2\\
        c_2+c'_2=k-a_1\\
        c_3+c'_3=a_1+a_2-k}}
    X_{\pseudopartition{(c_1,c_2,c_3)}}\otimes X_{\pseudopartition{(c_1',c_2',c_3')}},
  \end{equation}
  By collecting together the terms
  in~\eqref{E:psi-comultiplicative-RHS} with $n+m=k$ and interchanging
  the sums, it is easy to see that~\eqref{E:psi-comultiplicative-RHS}
  and~\eqref{E:psi-comultiplicative-LHS} are the same expression.
\end{proof}

\begin{corollary}
  \label{C:iso-sigma}
  The map $(\sym,\externalsym,\Delta)\to (\sym,\smashprod,\Delta)$ given by
  \begin{equation}
    \label{E:iso-sigma-sym}
    h_{(a_1,\ldots,a_r)}
    \mapsto 
    h_{(a_1)}\smashprod\cdots\smashprod h_{(a_r)}
  \end{equation}
  is an isomorphism of Hopf algebras (which does not preserve the gradings).
\end{corollary}


The Heisenberg and internal products are also isomorphic at the level of
$\symcomplete$.




\begin{theorem}
  \label{T:smash-comp}
  The map $(\symcomplete,\smashprod)\to (\symcomplete,\internalsym)$ given by
  \begin{equation}
    \label{E:smash-comp}
    f\mapsto f\externalsym \sum_{n\geq 0} h_{(n)} 
  \end{equation}
  is an isomorphism of algebras.
\end{theorem}

\begin{proof}
  This isomorphism follows from the
  isomorphism~\eqref{E:isomorphism-smash-composition} in the category
  of species.  Note that the species $\e$ corresponds to the object
  $(\trivialrep_0, \trivialrep_1,\dots)$ in the category
  $\representations$, where $\trivialrep_n$ is the trivial
  $S_n$-module.  Applying the Grothendieck group construction and then
  the Frobenius map $\frobenius$, we deduce that $\e$ maps into
  the element $ \sum_{n\ge0} h_{(n)} $ in $\symcomplete$.
\end{proof}

\begin{remark}
  \label{r:conv-smash-comp}
  We give negative answers to three questions on possible extensions of the above results. 
\begin{enumerate}
\item The isomorphism $(\Sigma,\externalsym)\cong (\Sigma,\smashprod)$ of
  Theorem~\ref{T:iso-sigma} does \emph{not} extend to an isomorphism
  between $(\widehat{\Sigma},\externalsym)$ and $(\widehat{\Sigma},\smashprod)$.
  Indeed, in case it did extend: 
\begin{equation}\label{E:nonextensionpsi}
    \psi(X_{(1)}+X_{(1,1)}+X_{(1,1,1)}+\cdots)= X_{(1)}+X_{(1)}\smashprod X_{(1)}+X_{(1)}\smashprod
    X_{(1)}\smashprod X_{(1)} +\cdots,
  \end{equation}
  each of the terms in the infinite sum appearing in the right hand side of \eqref{E:nonextensionpsi}, contributes with a factor of degree
  $1$ (namely, $X_{(1)}$); this infinite sum is not a
  well-defined element of $\widehat{\Sigma}$.
\item The isomorphism $(\symcomplete,\smashprod)\cong
  (\symcomplete,\internalsym)$ of Theorem~\ref{T:smash-comp} does not
  restrict to an isomorphism between $(\sym,\smashprod)$ and
  $(\sym,\internalsym)$.  Indeed, the element $1\in\sym$ maps to
  $\sum_{n\geq 0} h_{(n)}$ which is in $\symcomplete$ but not in $\sym$.
\item A similar isomorphism
  to~\eqref{E:smash-comp} cannot be established at the level of $\widehat\Sigma=\prod_{n\ge
    0}\Sigma_n$. The maps $\varphi:f\mapsto f\externalsym
\sum_{n\ge0}X_{(n)} \quad\text{and}\quad\psi:f\mapsto 
\sum_{n\ge0}X_{(n)}\externalsym f$ are not isomorphisms between $(\widehat\Sigma,\smashprod)$ and
  $(\widehat\Sigma,\internalsym)$ because they are not not multiplicative.
  
Indeed, using the
  rule~\eqref{E:smash-X} we obtain that:
  \begin{equation*}
    X_{(3)}\smashprod X_{(3)}
    =
    X_{(3)}+X_{(1,1,2)}+X_{(2,2,1)}+X_{(3,3)}.
  \end{equation*}
  Then:
  \begin{equation}
    \label{E:counterexample-LHS}
    \varphi(X_{(3)}\smashprod X_{(3)})
    =
    \sum_{n\ge 0} X_{\pseudopartition{(3,n)}}+
    \sum_{n\ge0} X_{\pseudopartition{(1,1,2,n)}}+
    \sum_{n\ge0} X_{\pseudopartition{(2,2,1,n)}}+
    \sum_{n\ge0} X_{\pseudopartition{(3,3,n)}}.
  \end{equation}
  On the other hand, computing $\varphi(X_{(3)})\internalsym
  \varphi(X_{(3)})$ using Solomon's rule, gives:
  \begin{align}
    \varphi(X_{(3)})\internalsym \varphi(X_{(3)})
    &=
    \sum_{n,m} X_{\pseudopartition{(3,n)}} \internalsym
    X_{\pseudopartition{(3,m)}}
    =
    \sum_nX_{\pseudopartition{(3,n)}} \internalsym
    X_{\pseudopartition{(3,n)}} \notag
    \\
    &=
    \sum_{n\ge 0}
    X_{\pseudopartition{(3,n)}} +
    X_{\pseudopartition{(2,1,1,n)}} +
    X_{\pseudopartition{(1,2,2,n)}} +
    X_{\pseudopartition{(3,3,n)}}.
    \label{E:counterexample-RHS}
  \end{align}
  We can see that~\eqref{E:counterexample-LHS}
  and~\eqref{E:counterexample-RHS} are different since, for example,
  the term $X_{\pseudopartition{(2,1,1,n)}}$ appears
  in~\eqref{E:counterexample-RHS} but there is no term
  in~\eqref{E:counterexample-LHS} whose index is a composition
  starting with $2,1,1$.  A similar argument can be applied to show that $\psi$ is also not multiplicative.
\end{enumerate}
\end{remark}


      \part{Quasi--symmetric functions}

           \section{The internal and external coproduct of quasi-symmetric functions}
\label{S:quasi-symmetric}

In this section we consider space of quasi-symmetric functions $\qsym$, 
dual to the non-commutative symmetric functions. First, in preparation to the  introduction by dualization a new coproduct: the \emph{Heisenberg coproduct} we recall the basic definitions of the internal and external coproducts.

In Sections \ref{S:ncsymmetric}, \ref{S:nc-to-commutative} and \ref{S:hopf} we dealt with the following diagram: 
\begin{equation*}
  \xymatrix@C=20pt@R=20pt{
    \Sigma
    \ar@{}[r]|-*[@]{\subseteq}
    \ar@{->>}[d]_(.4){\pi}
    &
    \Sinfty
    \\
    \sym
    &,
  }
\end{equation*}
\noindent
and endowed the different spaces with $\smashprod$, the Heisenberg product and $\Delta$, the external coproduct. In this part we introduce other coproducts, and to avoid confusions we rename the external coproduct as $\Delta_{\externalsym}$ (it was called simply $\Delta$ when considered in \eqref{E:coprodper} and \eqref{E:coproductoncompletehom}). We proved the compatibility of $(\smashprod, \Delta_{\externalsym})$ at the level of $\Sigma$ and $\Lambda$. This compatibility is consistent with the known results about of the external (or Malvenuto--Reutenauer) product and the coproduct at the three levels of the above diagram, and about the compatibility with the internal (and Solomon) product at the levels of $\Lambda$ and $\Sigma$ (compare with Remark~\ref{R:coproduct-in-S}).  
 
Of the three Hopf algebras: $(\Sigma,\externalsym,\Delta_{\externalsym})$,  $(\Sinfty,\externalperm,\Delta_{\externalsym})$ and  $(\Lambda,\externalsym,\Delta_{\externalsym})$ the second and third are self dual (see for example ~\cite{gelfand95,gessel84,hazewinkel03, hazewinkel05,malvenuto95}).

Hence, we can complete the picture adding  the space $\qsym$ that is the graded dual of $\Sigma$, and that will fit into the following
commutative diagram of Hopf algebras:
\begin{equation*}
  \xymatrix@C=20pt@R=20pt{
    \Sigma
    \ar@{}[r]|-*[@]{\subseteq}
    \ar@{->>}[d]_(.4){\pi}
    &
    \Sinfty
    \ar@{->>}[d]^(.4){F}
    \\
    \sym
    \ar@{}[r]|-*[@]{\stackrel{\pi^*}{\subseteq}}
    &
    \qsym
  }
\end{equation*}
The diagram is self dual with respect to the antidiagonal. The new maps
$F$ and $\pi^*$ are the duals of the inclusion of $\Sigma$ in $\Sinfty$ and of
the projection of $\Sigma$ onto $\sym$, respectively.  The map $F$ is
described in detail in~\cite{aguiar05} and will not be used here.
Note that the space $\Sinfty$ is a Hopf algebra  with respect to
the Malvenuto-Reutenauer product (as noted in
Remark~\ref{R:coproduct-in-S}) and not with respect to the composition product. 

Next we recall the definition of $\qsym$. 
Let $\alphabet{X}=\{x_1,x_2,\ldots\}$ be an {\em alphabet}, i.e. a countable set, totally ordered by $x_1<x_2<\cdots$. 
 Let $\field[[\alphabet{X}]]$ be the algebra of formal
power series on $\alphabet{X}$ and $\qsym=\qsym(\alphabet{X})$ the
subspace linearly spanned by the elements
\begin{equation}\label{E:def-M}
  M_\alpha
  =
  \sum_{i_1<\dots<i_r}x_{i_1}^{a_1}\cdots x_{i_r}^{a_r}
\end{equation}
as $\alpha=(a_1,\ldots,a_r)$ runs over all compositions of $n$, for
$n\geq 0$.  The space $\qsym$ is a graded subalgebra of
$\field[[\alphabet{X}]]$ known as the algebra of {\em quasi-symmetric
  functions} (see \cite{gessel84}).  It is clear that any symmetric function
is quasi-symmetric, hence we have the inclusion of algebras
$\sym \subseteq \qsym$.  In~\cite{malvenuto95} it is proved that this map is $\pi^*$ defined above as the dual of the projection $\pi:\Sigma\epi\sym$ (see Section~\ref{S:nc-to-commutative}).

The algebra $\qsym$ carries two coproducts $\Delta_{\internalsym}$ and
$\Delta_{\externalsym}$ which are defined via evaluation of quasi-symmetric
functions on alphabets.  Let $\alphabet{Y}$ be another alphabet.  We
can view the disjoint union $\alphabet{X}+\alphabet{Y}$ and the
Cartesian product $\alphabet{X}\times\alphabet{Y}$ as alphabets as
follows: on $\alphabet{X}+\alphabet{Y}$ we keep the ordering among the
variables of $\alphabet{X}$ and among the variables of $\alphabet{Y}$,
and we require that every variable of $\alphabet{X}$ precede every
variable of $\alphabet{Y}$.  On $\alphabet{X}\times\alphabet{Y}$ we
impose the reverse lexicographic order:
\begin{equation*}
  (x_h,y_i)\leq (x_j,y_k)
  \qquad\text{means}\qquad
  y_i<y_k \text{ or } (y_i=y_k \text{ and }x_h< x_j).
\end{equation*}

The coproducts are defined by the formulas
\begin{equation*}
  \Delta_{\internalsym}
  \bigl( f(\alphabet{X})\bigr)
  =
  f(\alphabet{X}\times\alphabet{Y}) 
  \quad\text{and}\quad
  \Delta_{\externalsym}
  \bigl(f(\alphabet{X})\bigr)
  =
  f(\alphabet{X}+\alphabet{Y}),
\end{equation*}
together with the identification
$
\qsym(\alphabet{X},\alphabet{Y})\cong
\qsym(\alphabet{X})\otimes\qsym(\alphabet{X})
$
(separation of variables).

Consider the following pairing between the homogeneous components of
degree $n$ of $\qsym$ and $\Sigma$:
\begin{equation}
  \label{E:pairing}
  \langle M_\alpha, X_\beta\rangle
  =
  \delta_{\alpha,\beta}.
\end{equation}
It is known~\cite{gelfand95,gessel84,malvenuto95} that this pairing
identifies the product of quasi-symmetric functions with the
coproduct~\eqref{D:coproduct-sigma} of $\Sigma$, and the coproducts
$\Delta_{\internalsym}$ and $\Delta_{\externalsym}$ with the internal and external
products of $\Sigma$.  In other words,
\begin{equation*}
  \langle fg, u\rangle 
  =  
  \langle f\otimes g, \Delta(u) \rangle,
  \quad
  \langle \Delta_{\internalsym}f, u\otimes v\rangle 
  = 
  \langle f, u v\rangle,
  \quad
  \langle \Delta_{\externalsym}f, u\otimes v\rangle 
  = \langle f, u\externalsym v\rangle,
\end{equation*}
for any $f,g\in\qsym$ and $u,v\in\Sigma$.  Here we set $\langle
f\otimes g, u\otimes v\rangle=\langle f,u\rangle\langle g,v\rangle$.
\section{The Heisenberg coproduct of quasi--symmetric functions}
Let $\Delta_{\smashprod}$ be the coproduct of $\qsym$ dual to the Heisenberg product of $\Sigma$:
\begin{equation*}
  \langle \Delta_{\smashprod}f, u\otimes v\rangle 
  = 
  \langle f, u\smashprod v\rangle.
\end{equation*}
Since the Heisenberg product is a sum of terms of various
degrees~\eqref{E:smash-degree}, the Heisenberg coproduct is a finite sum of
the form
\begin{equation*}
  \Delta_{\smashprod}(f)
  =\sum_i f_i\otimes f'_i
\end{equation*}
with $0\leq \deg(f_i)$ and $\deg(f'_i)\leq \deg(f)\leq
\deg(f_i)+\deg(f'_i)$.  The terms corresponding to $\deg(f)=
\deg(f_i)=\deg(f'_i)$ and to $\deg(f)=\deg(f_i)+\deg(f'_i)$ are the
coproducts $\Delta_{\internalsym}(f)$ and $\Delta_{\externalsym}(f)$, respectively.

Let $\alphabet{1}+\alphabet{X}$ denote the alphabet $\alphabet{X}$
together with a new variable $x_0$ smaller than all the others and
with the property $x_0^k=x_0$ for any natural $k$.  Let
\begin{equation*}
  (\alphabet{1}+\alphabet{X})\times(\alphabet{1}+\alphabet{Y})-\alphabet{1}
\end{equation*}
be the Cartesian product of the alphabets $\alphabet{1}+\alphabet{X}$
and $\alphabet{1}+\alphabet{Y}$ with reverse lexicographic ordering
and with the variable $(x_0,y_0)$ removed.  We can suggestively denote
$(\alphabet{1}+\alphabet{X})\times(\alphabet{1}+\alphabet{Y})-\alphabet{1}$
by $\alphabet{X}+\alphabet{Y}+\alphabet{X}\alphabet{Y}$, although the
order is given properly by the former expression.

The following result was obtained in conversation with Arun Ram.

\begin{theorem}
  \label{T:coprod-Q} 
  For any $f\in\qsym$,
  \begin{equation*}
    \Delta_{\smashprod}
    \bigl(f(\alphabet{X})\bigr)
    =
    f(
    \alphabet{X}+\alphabet{Y}+\alphabet{X}\alphabet{Y}
    ).
  \end{equation*}
\end{theorem}


\begin{proof}
We have to show that, with respect to the pairing~\eqref{E:pairing},
\begin{equation} 
  \label{E:pairing-smash}
  \langle 
  M_\gamma(\alphabet{X}+\alphabet{Y}+\alphabet{X}\alphabet{Y}), 
  X_\alpha\otimes X_\beta\rangle 
  = 
  \langle
  M_\gamma, X_\alpha\smashprod X_\beta
  \rangle
\end{equation}
for all $\gamma$, $\alpha$ and $\beta$ compositions of $n$, $p$ and
$q$, respectively.  Let us fix a composition $\gamma$ of $n$ and let
$k$ the length of $\gamma$.  Denote the set of indices of
$M_\gamma(\alphabet{X}+\alphabet{Y}+\alphabet{X}\alphabet{Y})$ by
\begin{equation*}
  \specialset{Y} 
  = 
  \bigl\{ 
  \bigl( (i_1,j_1),\ldots, (i_k,j_k) \bigr) 
  \bigm|
  (i_1,j_1) < \cdots < (i_k,j_k) 
  \bigr\}.
\end{equation*}
Consider the set
$
\specialset{A}_{\alpha,\beta}
=
\bigl\{M\in\M_{\alpha,\beta}^n\mid
w(M)=\gamma\bigr\}
$
and define the map
\begin{equation*}
  \psi:\specialset{Y} 
  \to 
  \bigcup_{\alpha,\beta}
  \specialset{A}_{\alpha,\beta}
\end{equation*}
as follows: given $(i_1,j_1)<\cdots<(i_k,j_k)$, let $\widetilde
M=(\widetilde m_{ij})$ be a matrix of zeros big enough to set
$\widetilde m_{j_\ell i_\ell}=\gamma_\ell$ (as usual in these proofs,
we start the indices in 0).  Then, remove all zero rows and columns,
except those with index 0; let us call $M$ to the result.  Since
$(0,0)$ is not a possible index, we have $m_{00}=0$.  Thus,
$M\in\specialset{M}_{\alpha,\beta}^n$ where $\alpha$ is the
composition obtained by adding all the rows of $M$ but the first, and
analogously with $\beta$ and the rows of $M$.

The map $\psi$ is surjective, since, given some
$M\in\specialset{A}_{\alpha,\beta}$, we can build a sequence of
indices in $\specialset{Y}$ by reading the nonzero entries of $M$, say
$m_{uv}$, and considering the pairs $(v,u)$ lexicographically
ordered.  Therefore, we can write
\begin{equation*}
  M_\gamma(\alphabet{X}+\alphabet{Y}+\alphabet{X}\alphabet{Y}) 
  =
  \sum_{q\in\specialset{Y}}(xy)_q^\gamma
  = 
  \sum_{\alpha,\beta}
  \sum_{M\in\specialset{A}_{\alpha,\beta}} 
  \sum_{q\in \psi^{-1}(M)}
  (xy)_q^\gamma
\end{equation*}
where $(xy)_q^\gamma$ denotes the monomial
$(x_{i_1}y_{i_1})^{\gamma_1}\cdots (x_{i_k}y_{i_k})^{\gamma_k}$ for
$q= \bigl((i_1,j_1),\ldots,(i_k,j_k)\bigr)$.  Collecting together the
$x$'s and $y$'s establishes a bijection between the terms of the last
sum indexed over $\psi^{-1}(M)$ and the terms of
$M_\alpha(\alphabet{X})M_\beta(\alphabet{Y})$.  Indeed, take a term
from this product, given by indices $i_{r_1}<\cdots<i_{r_k}$ and
$j_{s_1}<\cdots<j_{s_\ell}$, and build the pairs $(j_{s_u},i_{r_v})$
such that $m_{v,u}\not=0$.  We also have to consider the pairs
$(0,i_{r_v})$ and $(j_{s_u},0)$ according to nonzero entries in the
first row and column of $M$.  Ordering these indices it is clear that
they belong to $\psi^{-1}(M)$ and this is the inverse process of
grouping $x$'s and $y$'s.

Then, we can write
\begin{equation*}
  M_\gamma(\alphabet{X}+\alphabet{Y}+\alphabet{X}\alphabet{Y})
  =
  \sum_{\alpha,\beta}
  \sum_{M\in\specialset{A}_{\alpha,\beta}}
  M_\alpha(\alphabet{X})M_\beta(\alphabet{Y}) 
  = 
  \sum_{\alpha,\beta}\card \specialset{A}_{\alpha,\beta}
  M_\alpha(\alphabet{X}) M_\beta(\alphabet{Y}).
\end{equation*}
which obviously implies the equation~\eqref{E:pairing-smash}.
\end{proof}

We can express the dual of the isomorphism in
Theorem~\ref{T:iso-sigma} in term of alphabets in the full dual of
$\Sigma$, which is $\qsymcomplete=\prod_{n\ge 0}\qsym_n$.  The pairing
$\langle\,,\rangle: \Sigma\times\qsymcomplete\to\field$ is defined by
\begin{equation*}
  \langle f, g\rangle
  =
  \sum_n \langle f_n,g_n\rangle_n
\end{equation*}
where $f_n$ and $g_n$ are the restrictions of $f$ and $g$ to the
homogeneous components of degree $n$, and $\langle\,,\rangle_n$ is the
pairing defined in~\eqref{E:pairing}.

For this, given an alphabet $\alphabet{X}$ we define its exponential,
$\expalphabet{X}$, by
\begin{equation*}
  \expalphabet{X}
  =
  \alphabet{X}+\alphabet{X}^{(2)} + \alphabet{X}^{(3)} +\cdots
\end{equation*}
where the {\em divided power} $\alphabet{X}^{(n)}$ is the set
\begin{equation}
  \label{E:definition-alphabet-exponential}
  \alphabet{X}^{(n)}
  =
  \bigl\{
  (x_{i_1},x_{i_2},\ldots,x_{i_n})\in\alphabet{X}^n
  \bigm|
  x_{i_1}<x_{i_2}<\cdots<x_{i_n}
  \bigr\}.
\end{equation}
We endow $\expalphabet{X}$ with the reverse lexicographic order.  With
this notations the following equation holds:
\begin{equation*}
  \expalphabet{X+Y} 
  = 
  \bigl(1+\expalphabet{X}\bigr) \bigl(1+\expalphabet{Y}\bigr) - 1
\end{equation*}
where the equality is considered as ordered sets.  Indeed, denote by
$(x)_k$ the monomial $x_{i_1}\cdots x_{i_k}$ with $i_1<\cdots <i_k$.
Then, given $(x)_k(y)_\ell < (x')_{k'}(y')_{\ell'}$ in
$\expalphabet{X+Y}$, it is immediate to see that either
$(y)_\ell<(y')_{\ell'}$ or $(y)_\ell=(y')_{\ell'}$ and
$(x)_k<(x')_{k'}$, which is the definition of the order in the left
hand side.  Clearly, the same argument applies in the other direction.

\begin{theorem} 
  \label{T:dual-iso-sigma}
  The dual of the isomorphism $\psi$ from
  $(\Sigma,\externalsym,\Delta_{\externalsym})$ to $(\Sigma,\smashprod,\Delta_{\externalsym})$ of
  Theorem~\ref{T:iso-sigma} with respect to the pairing
  $\langle\,,\rangle$, is the isomorphism $\psi^*$ from
  $(\qsymcomplete,\cdot,\Delta_{\smashprod})$ to
  $(\qsymcomplete,\cdot,\Delta_{\externalsym})$ given by
  \begin{equation*}
    \psi^*(f)=f\bigl(\expalphabet{X}\bigr).
  \end{equation*}
\end{theorem}

\begin{proof} We have to show that $\bigl\langle
  \psi(X_\gamma),f\bigr\rangle= \bigl\langle
  X_\gamma,\psi^*(f)\bigr\rangle$.  Observe that, from the definition
  of the pairing, it is enough to prove this equation for each
  degree.  Moreover, it is enough to prove it for the generators of the
  algebra $(\Sigma,\externalsym)$ since, for $g$ and $g'$ generators
\begin{align*}
  \bigl\langle\psi(g\externalsym g'),f\bigr\rangle 
  &= 
  \bigl\langle
  \psi(g)\smashprod\psi(g'), f\bigr\rangle 
  \\
  &=
  \bigl\langle 
  \psi(g)\otimes\psi(g'), f(\alphabet{X+Y+XY})
  \bigr\rangle 
  \\
  &= 
  \sum_i 
  \bigl\langle\psi(g),f_i(\alphabet{X})
  \bigr\rangle\;
  \bigl\langle
  \psi(g'),f'_i(\alphabet{Y})
  \bigr\rangle 
  \\
  &=
  \sum_i 
  \bigl\langle g, f_i\bigr(\expalphabet{X} \bigl) 
  \bigr\rangle\;
  \bigl\langle g', f'_i\bigr(\expalphabet{Y}\bigl) 
  \bigr\rangle  
  \\
  &= 
  \bigl\langle g\otimes g',
  f\bigl(\expalphabet{X}+\expalphabet{Y}+\expalphabet{X}
  \expalphabet{Y}
  \bigr)
  \bigr\rangle 
  \\
  &=
  \bigl\langle g\otimes g', f\bigl(\expalphabet{X+Y}\bigr)
  \bigr\rangle 
  \\
  &=
  \bigl\langle g\otimes g', \Delta\bigl(f(\expalphabet{X})\bigr) 
  \bigr\rangle
  \\
  &= 
  \bigl\langle g\externalsym g', f\bigl(\expalphabet{X}\bigr)
  \bigr\rangle.
\end{align*}
Next, we prove the duality for the set of generators $X_{(n)}$ for $n\ge 0$, and for $f=M_\alpha$ where $\alpha$
is a composition of $n$.  In this case we have $\psi(X_{(n)})=X_{(n)}$
and the equation $\bigl\langle X_{(n)},M_\alpha\bigr\rangle =
\bigl\langle X_{(n)}, M_\alpha\bigl(\expalphabet{X}\bigr) \bigr\rangle
= \delta_{(n),\alpha}$ is immediately verified.
\end{proof}

Endowed with the coproduct $\Delta_{\smashprod}$, the algebra $\qsym$ is
a graded connected Hopf algebra, in duality with the graded connected
Hopf algebra $(\Sigma,\smashprod,\Delta)$.  

We finish this Section by expressing the antipode of
this Hopf algebra in terms of the alphabets.

First, define the evaluation of quasi-symmetric functions on the the
opposite of an alphabet $\alphabet{X}$ by the equation
\begin{equation}
  \label{E:definition-alphabet-negative}
  M_\alpha(-\alphabet{X}) 
  = 
  (-1)^r \sum_{i_1\ge \cdots \ge i_r} x^{a_1}_{i_1}
  \cdots x^{a_r}_{i_r},
\end{equation}
for any composition $\alpha=(a_1,\ldots,a_r)$ (compare with the
definition of $M_\alpha$ in~\eqref{E:def-M}).  

Next, we define the
alphabet
\begin{equation}
  \label{E:definition-alphabet-star}
  \alphabet{X}^*
  =
  \alphabet{X} +
  \alphabet{X}^2 +
  \alphabet{X}^3 + \cdots
\end{equation}
as the disjoint union of the Cartesian powers $\alphabet{X}^n$ under
reverse lexicographic order.  For instance
$(x_3,x_1,x_2)<(x_2,x_2)<(x_1,x_3,x_2)$.

\begin{theorem}
  \label{T:antipode-Q}
  The antipode of the Hopf algebra of quasi-symmetric functions
  $(\qsym,\cdot,\Delta_{\smashprod})$ is:
  \begin{equation*}
    S_{\smashprod}(f) = f\bigl((-\alphabet{X})^*\bigr).
  \end{equation*}
\end{theorem}

\begin{proof}
  By Theorem~\ref{T:coprod-Q}, it is enough to prove that
  $M_\alpha\bigl(\alphabet{X}+(-\alphabet{X})^* + \alphabet{X}
  (-\alphabet{X})^* \bigr)=0$ for any alphabet $\alphabet{X}$ and for
  any composition $\alpha$.  We ilustrate the argument for a
  composition with only one part: $\alpha=(a)$, the argument for a
  composition with several parts is essentially the same.

  By selecting variables from each of the three alphabets
  $\alphabet{X}$, $(-\alphabet{X})^*$, and $\alphabet{X}
  (-\alphabet{X})^*$, we can write
  \begin{multline*}
    M_{(a)}\bigl(\alphabet{X}+(-\alphabet{X})^* + \alphabet{X}
    (-\alphabet{X})^* \bigr)
    \\
    =
    \sum x_i^a +
    \sum_r (-1)^r \sum (x_{i_1}\cdots x_{i_r})^a +
    \sum_r (-1)^r \sum x_j^a (x_{i_1}\cdots x_{i_r})^a.
  \end{multline*}
  It is easy to see that the first sum cancel with the terms with
  $r=1$ of the second sum, while the remaining terms of the second sum
  cancel with the last sum.
\end{proof}

      \part{Appendix}

           \section{The proofs} 
\label{S:appendix}

In this Appendix we provide the postponed proofs of the technical lemmas used in the paper. 

\subsection{Proof of Lemma~\ref{L:bijection-for-ha-smash-hb}}
\label{P:bijection-for-ha-smash-hb}

\begin{proof}
  To define the bijection $\Upsilon\to \M_{\alpha,\beta}^n$, we start
  by splitting the interals $[1,p]$ and $[1,q]$ as below:
  \begin{alignat*}{2}
    E_1
    &=
    [1,a_1],
    &\qquad
    F_1
    &=
    [1,b_1],
    \\
    E_2
    &=
    [a_1+1, a_1+a_2],
    &\qquad
    F_2
    &=
    [b_1+1, b_1+b_2],
    \\
    &\centerinequal{\vdots}
    &\qquad
    &\centerinequal\vdots
    \\
    E_k
    &=
    [a_1+\dots+a_{k-1}+1, p],
    &\qquad
    F_s
    &=
    [b_1+\dots+b_{s-1}+1, q],
  \end{alignat*}
  where $\alpha=(a_1,\dots,a_k)$ and $\beta=(b_1,\dots,b_s)$.  Given an
  element $v=\sigma\times\tau \in \S_p\times\S_q$ we consider the
  shuffles $\zeta_\alpha(\sigma)\in\Sh(\alpha)$ and
  $\zeta_\beta(\tau)\in\Sh(\beta)$ characterized by the equations
  \begin{equation}
    \label{E:shuffles-representatives}
    \sigma
    =
    \zeta_\alpha(\sigma) u,
    \qquad
    \tau
    =
    \zeta_\beta(\tau) v,
  \end{equation}
  with $u\in\S_\alpha$ and $v\in\S_\beta$.  To simplify the notation,
  we write $\zeta_\alpha=\zeta_\alpha(\sigma)$ and
  $\zeta_\beta=\zeta_\beta(\tau)$.  We further split each interval
  $E_i$ and $F_j$ as below:
  \begin{equation*}
    E_i = E'_i \sqcup E''_i,
    \qquad
    F_j = F'_j \sqcup F''_j,
  \end{equation*}
  such that
  \begin{alignat*}{2}
    \zeta_\alpha(E'_i) 
    &\subseteq
    [1,n-q],
    &\qquad
    \zeta_\beta(F'_j)
    &\subseteq
    [1,p+q-n],
    \\
    \zeta_\alpha(E''_i) 
    &\subseteq
    [n-q+1,p],
    &\qquad
    \zeta_\beta(F''_j)
    &\subseteq
    [p+q-n+1,q],
  \end{alignat*}
  for $i=1,\dots,k$ and $j=1,\dots,s$.  Observe that with these
  definitions we have the decomposition of the interval $[1,n]$ into
  \begin{align}
    \label{E:decomposition-1-n-q}
    [1,n-q] 
    &=
    \bigsqcup_{i=1}^k \zeta_\alpha(E'_i),
    \\ \label{E:decomposition-n-q-p}
    [n-q+1,p]
    &=
    \bigsqcup_{i=1}^k \zeta_\alpha(E''_i)
    =
    \bigsqcup_{j=1}^s \bigl(n-q +\zeta_\beta(F'_j)\bigr),
    \\ \label{E:decomposition-p-n}
    [p+1,n]
    &=
    \bigsqcup_{j=1}^s \bigl(n-q +\zeta_\beta(F''_j)\bigr).
  \end{align}
  Define the matrix $M_{\sigma\times\tau}$ of dimension $(k+1)\times(s+1)$
  whose entries are
  \begin{alignat*}{2}
    m_{00} &= 0, &&
    \\
    m_{i0} &= \card{E'_i},
    &\qquad&
    \text{for $i=1,\dots,k$,}
    \\
    m_{0j} &= \card{F''_j},
    &\qquad&
    \text{for $j=1,\dots,s$,}
    \\
    m_{ij} &= 
    \card{\Bigl[ \zeta_\alpha(E''_i) \cap
      \bigl(n-q+\zeta_\beta(F'_j)\bigr) \Bigr]}
    &\qquad&
    \text{otherwise.}
  \end{alignat*}
  
  The matrix $M_{\sigma\times\tau}$ belongs to $\M_{\alpha,\beta}^n$.
  Assume that $i\not=0$.  Since $\zeta_\alpha(E''_i)\subseteq
  [n-q+1,p]\subseteq \bigsqcup_{j=1}^s \bigl(n-q
  +\zeta_\beta(F''_j)\bigr)$, we get
  \begin{align*}
    \sum_{j=0}^s m_{ij}
    &=
    \card{E'_i} + \sum_{j=1}^s 
    \card{\Bigl[ \zeta_\alpha(E''_i) \cap
      \bigl(n-q+\zeta_\beta(F'_j)\bigr) \Bigr]}
    \\
    &=
    \card{E'_i} + \card\Bigl[
      \zeta_\alpha(E''_i)\cap \bigsqcup_{j=1}^s
      \bigl(n-q+\zeta_\beta(F'_j)\bigr)
    \Bigr]
    \\
    &=
    \card{E'_i} + \card\bigl(\zeta_\alpha(E''_i) \bigr) = \card{E_i} = a_i.
  \end{align*}
  On the other hand, if $i=0$, then, by~\eqref{E:decomposition-p-n},
  the sum of $m_{0j}$ for $j=0,\dots,s$, coincides with
  $\card[p+1,n]=n-p$.

  Next we show that the matrix $M_{\sigma\times\tau}$ does not depend on
  the choice of representative of the coset
  $(\S_p\times_n\S_q)v$. Let $x\in\S_{n-q}$, $y\in\S_{p+q-n}$, and
  $z\in\S_{n-p}$, so that $x\times y\times z\in
  \S_p\times_n\S_q$. Consider the representative $v'=\sigma'\times\tau'$
  where
  \begin{equation*}
    \sigma'
    =
    (x\times y) \sigma
    \quad\text{and}\quad
    \tau'
    =
    (y\times z) \tau.
  \end{equation*}
  Let $\zeta'_\alpha$ and $\zeta'_\beta$ the shuffles associate to
  $v'$. As $\zeta_\alpha(E'_i)\subseteq[1,n-q]$, then $(x\times y)
  \bigl(\zeta_\alpha(E'_i)\bigr) =
  x\bigl(\zeta_\alpha(E'_i)\bigr)$. But we also have
  $x\bigl(\zeta_\alpha(E'_i)\bigr) =
  \zeta'_\alpha(E_i)=\zeta'_\alpha(\tilde
  E'_i)\sqcup\zeta'_\alpha(\tilde E''_i)$, where $E_i=\tilde E'_i\sqcup
  \tilde E''_i$ is the decomposition of $E_i$ corresponding to the
  shuffle $\zeta'_\alpha$, that satisfies $\zeta'_\alpha(\tilde E'_i)
  \subseteq [1,n-q]$ and $\zeta'_\alpha(\tilde E''_i)
  \subseteq [n-q+1,p]$. In summary, $x\bigl(\zeta_\alpha(E'_i)\bigr)
  \subseteq \zeta'_\alpha(\tilde E'_i)$. Interchanging the roles of
  $\zeta_\alpha$ and $\zeta'_\alpha$ we obtain an equality, which
  implies that $m_{i0}=\card E'_i = \card \tilde E'_i = m'_{i0}$,
  where $m'_{ij}$ are the entries of the matrix $M_{v'}$. This proves
  the equality of the first row of the matrices. The argument for the
  other rows is similar.

  The matrix $M_v$ do not depend on the choice of representative of
  $v(\S_\alpha\times \S_\beta)$, since the shuffles
  satisfying~\eqref{E:shuffles-representatives} are the same for all
  the elements on this coset. In conclusion, the matrix $M_v$ depends
  only on the double cosets
  $(\S_p\times_n\S_q)v(\S_\alpha\times\S_\beta)$.

  Next we show that the parabolic subgroup $\S_{p(M_v)}$ is
  $\S_\alpha\times_n^\upsilon\S_\beta$. An element of
  $\S_\alpha\times_n^\upsilon\S_\beta$ can be written as $x\times
  y\times z$ where
  \begin{align*}
    x\times y 
    &= 
    \zeta_\alpha(\sigma_{a_1}\times\dots\times\sigma_{a_k})\zeta^{-1}_\alpha,
    \\
    y\times z
    &=
    \zeta_\beta(\tau_{b_1}\times\dots\times\tau_{b_s})\zeta^{-1}_\beta.
  \end{align*}
  Evaluating at $\zeta_\alpha(E'_i)$ we deduce that
  $\zeta_\alpha\sigma_{a_i}(E'_i) = x(E'_i)$ and conclude that
  $\sigma_{a_i}(E'_i)=E'_i$. Proceeding in a similar manner with the
  other decompositions we obtain
  \begin{alignat}{2}
    \label{E:first-decomposition}
    \sigma_{a_i}(E'_i) &= E'_i,
    &\qquad
    \tau_{b_j}(F'_j) &= F'_j,
    \\
    \sigma_{a_i}(E''_i) &= E''_i,
    &\qquad
    \tau_{b_j}(F''_j) &= F''_j,
  \end{alignat}
  for all $i=1,\dots,k$ and $j=1,\dots,s$.

  This decomposition can be further refined. Evaluating as above at
  the subsets $X_{ij}=\zeta_\alpha(E''_i)\cap\zeta_\beta(F'_j)$, we
  obtain the equality
  \begin{equation*}
    \zeta_\alpha\sigma_{a_i}
    \bigl(
       \zeta_\alpha^{-1}(X_{ij})
    \bigr)
    =
    y(X_{ij})
    =
    \zeta_\beta\tau_{b_j}
    \bigl(
       \zeta_\beta^{-1}(X_{ij})
    \bigr).
  \end{equation*}
  Now, 
  $
    \zeta_\alpha\sigma_{a_i}
    \bigl(
       \zeta_\alpha^{-1}(X_{ij})
    \bigr)
    \subseteq \zeta_\alpha(E''_i)
  $
  and also
  $
    \zeta_\beta\tau_{b_j}
    \bigl(
       \zeta_\beta^{-1}(X_{ij})
    \bigr)
    \subseteq \zeta_\beta(F'_j)
  $.
  From the above equality we conclude that 
  $
  \zeta_\alpha\sigma_{a_i}
    \bigl(
       \zeta_\alpha^{-1}(X_{ij})
    \bigr)
    \subseteq \zeta_\alpha(E''_i)\cap\zeta_\beta(F'_j)
  $,
  and then 
  $
  \sigma_{a_i}
  \bigl(
     \zeta_\alpha^{-1}(X_{ij})
  \bigr) \subseteq \zeta_\alpha^{-1}(X_{ij})
  $.
  This inclusion is actually an equality, since both sets have the
  same cardinality. Therefore, we get the following refinment
  of~\eqref{E:first-decomposition}
  \begin{alignat*}{2}
    \sigma_{a_i}(E'_i) &= E'_i,
    &\qquad
    \sigma_{a_i}\bigl(\zeta^{-1}_\alpha(X_{ij})\bigr) &= \zeta^{-1}_\alpha(X_{ij}),
    \\
    \tau_{b_j}(F''_j) &= F''_j,
    &\qquad
    \tau_{b_j}\bigl(\zeta^{-1}_\beta(X_{ij})\bigr) &= \zeta^{-1}_\beta(X_{ij}).
  \end{alignat*}
  Note that $\card X_{ij}=m_{ij}$, and thus the previous decomposition
  shows that $x\times y\times z$ belongs to $\S_{p(M)}$.

  The map $\upsilon\mapsto M_\upsilon$ is invertible, since from the
  entries of the matrix $M_\upsilon$ we can recover the shuffles
  $\zeta_\alpha$ and $\zeta_\beta$, which are in the same double coset
  as $\upsilon$.
\end{proof}

\subsection{Proof of Lemma~\ref{L:interval}}
\label{P:interval}

\begin{proof}
  As $\eta$ and $\tau$ are fixed throughout this lemma, we write
  $\varphi=\varphi_{\eta,\tau}$.  Let $x,y\in F_i\cap \varphi^{-1}
  E_j$ with $x<y$.  Consider $z$ such that $x<z<y$.  Therefore,
  $x,y\in F_i$ and, since $F_i$ is an interval, we conclude that $z\in
  F_i$.

  On the other hand $\varphi(x), \varphi(y)\in E_j$.  Since $\tau\in
  \B_\beta$, then $\Id\times\tau$ is increasing in $F_i$:
  \begin{equation} 
    \label{E:inequality}
    (\Id\times\tau)(x) < (\Id\times\tau)(z) < (\Id\times\tau)(y).
  \end{equation}
  In order to prove that $\varphi(z)$ also belongs to $E_j$, we
  consider the following cases:
  \begin{enumerate}
  \item Assume that $j=0$.  Then, $\varphi(x),\varphi(y)\in
    E_0=[p+1,n]$.  Since $(\eta\times\Id)$ is the identity on that
    interval, this implies that $\beta_0(\Id\times\tau)(x)$ and
    $\beta_0(\Id\times\tau)(y)$ are in $[p+1,n]$.  But
    $\beta_0^{-1}[p+1,n] = [n-q+1, 2n-p-q]$ and $\beta_0$ is
    increasing in that set.  Therefore, the three terms
    in~\eqref{E:inequality} belong to $[n-q+1,2n-p-q]$ and, applying
    $(\eta\times\Id)\beta_0$, which is increasing on this set, we
    obtain that $\varphi(x)<\varphi(z)<\varphi(y)$.
  \item Assume that $j>0$.  Consider the cases:
    \begin{enumerate}
    \item Assume $i=0$.  In this case we have $x,z,y\in F_0=[1,n-q]$.
      Then, applying $\Id\times\tau|_{F_0}=\Id$ we continue in the
      same set. The permutation $\beta_0$ sends increasingly $[1,n-q]$
      into $[p+q-n+1,p]$.  In this last interval, $\eta$ is also
      increasing. Thus, the inequality~\eqref{E:inequality} implies
      that $\varphi(x)<\varphi(z)< \varphi(y)$.
    \item Assume $i>0$.  We have that $x,y,z\in F_j\subset [n-q+1,n]$.
      Applying $\Id\times\tau$ we have that the terms
      of~\eqref{E:inequality} are also in $[n-q+1,n]$.  If
      $(\Id\times\tau)(x) \in [n-q+1,2n-p-q]$, then
      $\beta_0(\Id\times\tau)(x)\in [p+1,n]$ and
      $\varphi(x)\in[p+1,n]=E_0$, which contradicts the assumption
      $j>0$.  Therefore, the terms in~\eqref{E:inequality} belong to
      $[2n-p-q+1, n]$.  The permutation $\beta_0$ maps increasingly
      this interval into $[1,p+q-n]$, and $\eta$ is also increasing in
      that image.  Thus, we conclude that
      $\varphi(x)<\varphi(z)<\varphi(y)$.
    \end{enumerate}
  \end{enumerate}
  In all the cases we obtain that $\varphi(x)<\varphi(z)<\varphi(y)$,
  and since $\varphi(x)$ and $\varphi(y)$ belong to the interval
  $E_j$, we deduce that $\varphi(z)\in E_j$. This proves that $F_i
  \cap \varphi^{-1}E_j$ is an interval.

  Notice that along the way we also proved that $\varphi$ is
  increasing in the intervals $F_i \cap \varphi^{-1}E_j$ as well as
  the assertions concerning the images.

  The fact that the intervals $F_i \cap \varphi^{-1}E_{j}$ are
  disjoint follows immediately from the fact that the sets $E_{j}$,
  for $j=0,\dots,r$, and the sets $F_{i}$, for $i=0,\dots,s$, are
  disjoint.  This finishes the proof.
\end{proof}

\subsection{Proof of Lemma~\ref{L:bijection}}
\label{P:bijection}

\begin{proof}
  For the matrix $M=\{m_{ij}\}$, denote by $s_{ij}$ the sum of the
  entries $m_{k\ell}$ of $M$ for $(k,\ell)\le (i,j)$ with respect to
  the lexicografical order of pairs.  We define $R_{00} = [1,s_{00}]$
  and $R_{ij} = [s_{k\ell}, s_{ij}]$ where $s_{ij}$ covers
  $s_{k\ell}$.  Observe that some of the intervals $R_{ij}$ may be
  empty.  Also note that $\card R_{ij}=m_{ij}$.

  The sequence $(R_{00},R_{01},\ldots, R_{sr})$ is a pseudo-partition
  of the interval $[n]$ and $\gamma\in\B_{c(M)}$ if and only if
  $\gamma$ is increasing in $R_{ij}$ for all $i\in\{0,\ldots ,s\}$ and
  $j\in\{0,\dots ,r\}$.

  Since $M\in\M_{\alpha,\beta}$ and therefore,
  $\sum_{j}\card(R_{ij})=\sum_j m_{ij} = \card F_i$, it follows that
  \begin{equation}
    \label{E:Fiunion}
    F_i = \bigcup_j R_{ij}.
  \end{equation}
  Moreover, if $\eta\in\Sh(p+q-n,n-q)$ and $\tau\in\B_\beta^\eta$,
  then $F_i\cap \varphi_{\eta,\tau}^{-1}E_j = R_{ij}$.  This can be
  seen from the fact both sets are intervals with the same cardinal
  and from the following relation:
  \begin{equation*}
    \bigcup_j (F_i\cap \varphi^{-1}_{\eta,\tau}E_j)
    =
    F_i
    =
    \bigcup_jR_{ij}.
  \end{equation*}
  In particular, we deduce that $\varphi_{\eta,\tau}$ is increasing in
  $R_{ij}$.

  Given $(\xi,\eta,\sigma,\tau)\in S_{\alpha,\beta}(M)$, we will show
  that $g_{\xi,\eta}(\sigma,\tau)\in\B_{c(M)}$.  To prove this, since
  $\varphi_{\eta,\tau}|_{R_{ij}}$ is increasing and
  $\varphi_{\eta,\tau}R_{ij} \subseteq E_j$, we observe that
  \begin{equation*}
    (\sigma\times\Id)(\eta\times\Id)\beta_0(\Id\times\tau)_{|_{R_{ij}}}
  \end{equation*}
  is also increasing.  According to Lemma~\ref{L:interval}, the images of
  $R_{ij}$ under the previous permutation are in $[1,p]$ or $[p+1,n]$,
  where $\xi$ is increasing.  Therefore, left multiplying by $\xi$ we
  deduce that $g_{\xi,\eta}(\sigma,\tau)$ is increasing in $R_{ij}$,
  which proves that it belongs to $\B_{c(M)}$.

  We prove now that $\psi$ is bijective. Given
  $\gamma\in\B_{w(M)}$, we show that there exists a unique
  quadruple $(\xi,\eta,\sigma,\tau)\in S_{\alpha,\beta}(M)$ such that
  $\psi(\xi,\eta,\sigma,\tau)=\gamma$.

  Assume there exists such a quadruple.  Using the fact that
  $E_j=\bigcup_i\varphi_{\eta,\tau}R_{ij}$, we deduce that
  \begin{equation} 
    \label{E:construct} 
    \xi(\sigma\times\Id)E_j 
    =
    \gamma
    \bigl(\bigcup_iR_{ij}\bigr).
  \end{equation}
  This proves the uniqueness of the permutation
  $\xi(\sigma\times\Id)$, in other words, it is the only permutation
  which maps $E_j$ increasingly into the set on the right side; and
  this implies the uniqueness of $\xi$ and $\sigma$.  Therefore, we
  have that
  $
  (\eta\times\Id)\beta_0(\Id\times\tau)
  =
  (\sigma\times\Id)^{-1}\xi^{-1}\gamma
  $.
  Thus, $\eta$ is characterized by the image of $[1,n-q]$ under the
  permutation on the right, which is $\eta[p+q-n+1,p]$.  The
  uniqueness of $\tau$ follows immediately.

  Given $\gamma\in\B_{c(M)}$, to construct $(\xi,\eta,\sigma,\tau)$ we
  note that
  \begin{equation} \label{E:inM}
    \card(E_j) 
    = 
    \sum_i m_{ij} 
    = 
    \card\Bigl(\bigcup_i R_{ij}\Bigr) 
    = 
    \card\biggl(
    \gamma\Bigl(\bigcup_i R_{ij}\Bigr) \biggr),
  \end{equation}
  and, thus, we can construct a permutation $\mu$ such
  that~\eqref{E:construct} is verified, increasingly mapping $E_j$
  into $\gamma\bigl(\bigcup_i R_{ij}\bigr)$.  This permutation can be
  written as $\mu=\xi(\sigma\times\mu')$ with $\xi\in\Sh(p,n-p)$,
  $\sigma\in\S_p$ and $\mu'\in\S_{n-p}$. Since $\mu$ is increasing on
  $E_0=[p+1,n]$ we conclude that $\mu'=\Id_{n-p}$, and from the
  monotony on $E_j$ with $j>0$ we deduce that $\sigma\in\B_\alpha$.
  In the same way as before, we construct $\eta$ by mapping the
  interval $[1,n-q]$ and for this, we will show that
  \begin{equation}
    \label{E:property}
    \text{$(\sigma\times\Id)^{-1}\xi^{-1}\gamma$ is
      increasing in $F_i$ for all $i$.}
  \end{equation}
  In particular, for $i=0$, we obtain the desired property to define
  $\eta$.  We then consider $\beta_0^{-1} (\eta\times\Id)^{-1}
  (\sigma\times\Id)^{-1}\gamma$, which equals $\Id\times\tau$ for some
  $\tau\in\S_p$.  Using \eqref{E:property} for $i>0$ we conclude
  that $\tau\in\B_\beta$; and it follows from ~\eqref{E:inM} that
  the constructed $\tau$ belongs to~$\B_\beta^\eta(M)$.

  It remains to prove~\eqref{E:property}.  Take $x_1,x_2\in F_i$ with
  $x_1<x_2$.  Then, $x_1\in R_{ij_1}$ and $x_2\in R_{ij_2}$ for some
  $j_1\le j_2$.  Assume that $j_1=j_2$, then $\gamma(x_1)<\gamma(x_2)$.
  In this case, we have $\gamma(x_1)=\xi(\sigma\times\Id)(e_1)$ and
  $\gamma(x_2)=\xi(\sigma\times\Id)(e_2)$ with $e_1,e_2\in E_j$.  Since
  $\sigma$ is increasing in $E_j$ we obtain that $e_1<e_2$ as desired.

  On the other hand, if $j_1<j_2$, then $e_1\in E_{j_1}$ and $e_2\in
  E_{j_2}$ and the conclusion follows easily as all the elements of
  $E_{j_1}$ are smaller than those of $E_{j_2}$.
\end{proof}


     \bibliographystyle{amsplain}
     \bibliography{biblio/biblio}

\end{document}